\documentclass[a4paper,10pt]{scrartcl}
\usepackage{a4,amssymb,epsfig}
\usepackage[all]{xy}
\usepackage{amsmath,amsthm}
\usepackage{lscape}
\usepackage[bookmarks,bookmarksnumbered,plainpages=false]{hyperref}
\usepackage{colonequals}
\usepackage{paralist}

\setdefaultenum{\normalfont (a)}{i.}{A.}{1.}

\newcommand{\R}{\mathbb{R}}
\newcommand{\N}{\mathbb{N}}
\newcommand{\norm}[1]{\left\lVert#1\right\rVert}

\newcommand{\SRVT}{\mathcal{R}}

\DeclareMathOperator{\id}{id}
\DeclareMathOperator{\Evol}{Evol}
\DeclareMathOperator{\A}{\Sigma}

\DeclareMathOperator{\Lf}{\mathbf{L}}

\DeclareMathOperator{\ev}{\mathrm{ev}}

\newcommand{\ver}{\mathrm{Vert}}
\newcommand{\hor}{\mathrm{Hor}}

\newcommand{\dd}{\mathrm{d}}
\newcommand{\coloneq}{\colonequals}
\newcommand{\pr}{\mathrm{pr}}
\newcommand{\pt}{\mathrm{pt}}

\newcommand{\AC}{AC}

\newtheorem{proposition}[subsection]{Proposition}
\newtheorem{lem}[subsection]{Lemma}

\newtheorem{thm}[subsection]{Theorem}
\theoremstyle{definition}
\newtheorem{setup} [subsection]{}
\newtheorem{defn} [subsection]{Definition}
\newtheorem{remark} [subsection]{Remark}

\title{Manifolds of absolutely continuous curves and the square root velocity framework}
\author{Alexander Schmeding\footnote{NTNU Trondheim, Email: \href{mailto:alexander.schmeding@math.ntnu.no}{alexander.schmeding@math.ntnu.no}}}
\date{\vspace{-5ex}}
\begin{document}

\maketitle

\begin{abstract}
 A classical result in Riemannian geometry states that the absolutely continuous curves into a (finite-dimensional) Riemannian manifold form an infinite-dimensional manifold. In the present paper this construction and related results are generalised to absolutely continuous curves with values in a strong Riemannian manifolds.
 
As an application we consider extensions of the square root velocity transform (SRVT) framework for shape analysis. Computations in this framework frequently lead to curves which leave the shape space (of smooth curves), and are only contained in a completion. In the vector valued case, this extends the SRVT to a space of absolutely continuous curves. We investigate the situation for shape spaces of manifold valued (absolutely continuous) curves.
\end{abstract}


\medskip


\tableofcontents

\section*{Introduction and statement of results} \addcontentsline{toc}{section}{Introduction and Statement of results}
 Many problems in object and pattern recognition \cite{bauer_overview_2014,bauer14cri,MR3265197,su14sao,srivastava11sao}, computer animation \cite{celledoni15sao} and signal analysis \cite{MR3442194} can be formulated in terms of similarities  of \emph{shapes}. By \emph{shape} we mean an unparametrized curve with values in a vector space or a manifold.
 Hence one studies equivalence classes of regular curves (i.e.\ smooth immersions, embeddings, absolutely continuous functions...) where equivalence is induced by reparametrisation. 
 To compare shapes, one places them in an infinite-dimensional manifold, the shape space. Comparisons in the shape space are then carried out by means of a distance function, usually induced by a Riemannian metric.
 
 There are many choices for a distance on shape spaces (see e.g.\ \cite{bauer_overview_2014} for a survey). We will concentrate on a certain first order Sobolev metric, which is considered in the square root velocity transform (SRVT) framework \cite{srivastava11sao}. This distance is well-suited for applications as it is related to the $L^2$-metric, which is easy to compute. 
 However, numerical computations in the SRVT framework frequently lead to objects which are not contained in shape space (of smooth mappings), but reside in a completion. For vector valued shapes this is the space of all absolutely continuous curves. Hence it is natural to consider the SRVT in this extended setting.
 A detailed analysis of the SRVT in the vector valued case can be found in \cite{bruveris2015}.
 Though originally developed for planar curves, the SRVT framework has since been expanded to (smooth) curves with values in Lie groups and Riemannian manifolds (see e.g.\ \cite{celledoni15sao,su14sao}).In this setting the completion should then be a manifold of absolutely continuous curves onto which the SRVT can be expanded.
 \medskip
 
 The purpose of the present paper is twofold: 
 As a first step, we construct manifolds of absolutely continuous curves with values in a (possibly infinite-dimensional) Riemannian manifold. This is a generalisation of a classical theory which was developed in the context of Riemannian geometry (but seems to be unknown to a larger audience). Many classical results generalise to our more general setting. In particular, we construct bundles of $L^p$-curve which are crucial to the second part of the paper. 
    
 In the second step, we analyse extensions of the SRVT framework on these manifolds. Our aim here is to prove that the SRVT (and its generalisations from the literature) always splits into a smooth part (given by derivation and translation) and a continuous (but non-differential) part (given by rescaling). In this sense the SRVT framework for manifolds generalises the results from \cite{bruveris2015}. 
 Unfortunately, it is so far unclear to the author, whether the extended SRVT on Riemannian manifolds can be inverted on manifold valued absolutely continuous curves (a sketch for a proof is supplied in \ref{setup: inverse SRVT}).
 However, for Lie group valued shape spaces, the situation is much better. Using some recent results from infinite-dimensional Lie theory, we obtain a full extension of the vector valued case.
 In particular, the SRVT on Lie group valued absolutely continuous curves is a homeomorphism which allows us to construct a metric on the completion by pullback of the $L^2$-distance. 
 \medskip
 
 We now discuss our results in greater detail. In this paper we deal with manifold valued absolutely continuous curves. For a Banach space $E$, a curve $\gamma \colon [a,b] \rightarrow E$ is called absolutely continuous (with derivative in $L^p$) if there is $\eta \in L^p([a,b],E)$ such that $\gamma (t) = \gamma (a) + \int_a^t \eta (s) \dd s$ (details on absolutely continuous curves are repeated in Appendix \ref{App: abscont}). Let now $M$ be a Banach manifold and $p\in \N$. Then we let $\AC^p (I,M)$ be the set of all absolutely continuous curves with values in $M$, i.e.\ which are locally in charts absolutely continuous curves (with derivatives in $L^p$). To construct the manifold $\AC^p (I,M)$, we need a \emph{strong Riemannian metric} $G$ on $M$.\footnote{A Riemannian manifold $(M,G)$ is called \emph{strong} if the topology induced by the inner product on each tangent space coincides with the natural topology. Note that this implies that $M$ is modelled on a Hilbert space. Every finite-dimensional Riemannian metric is strong. Infinite-dimensional manifolds can possess \emph{weak} Riemannian metrics, i.e.\ a smoothly varying non-degenerate bilinear form on $TM$, such that the induced topology on the tangent spaces is weaker than the natural topology (e.g.\ see \cite{MR3265197}). Not every Riemannian metric on a Hilbert manifold is strong, see e.g.\ \cite{0809.3104v1}.}. Then our results (Theorem \ref{thm: ACman}) subsume the following.
 \smallskip
 
 \textbf{Theorem A} \emph{Let $(M,G)$ be a strong Riemannian manifold. Then for every $p \in \N$ the absolutely continuous curves $\AC^p (I,M)$ can be made into a Banach manifold modelled on spaces of vector valued absolutely continuous functions. The construction is independent of the Riemannian metric.}\smallskip
 
 We remark that Theorem A is a classical result by Klingenberg and Flaschel (see \cite{FlaKli72,MR1330918}) in the case $p=2$ and $\dim M < \infty$. 
 The proof for the general case follows the strategy outlined in \cite{FlaKli72,MR1330918}.\footnote{Here \cite{FlaKli72} is by far the more detailed account of the construction. Unfortunately, this book seems to be available in German only.}
 However, we have to modify and augment the proof by avoiding compactness and certain bundle theoretic arguments.
 After the necessary ammendments have been made, the classical theory generalises in the wash ti the more general setting. 
 It does not seem possible to generalise Theorem A beyond the realm of strong Riemannian manifolds. 
 The reason for this is, that in the construction one has to use that the topology is finer than the compact open topology. This turns out to be a consequence of the Riemannian metric being strong.
 
 In recent investigations in infinite-dimensional Lie theory (see \cite{hgmeasure16}), Lie groups of absolutely continuous curves with values in a (possibly infinite-dimensional) Lie group were considered.
 It is now straight forward to relate these constructions to the manifold structure from Theorem A:
 Recall that every Hilbert Lie group is a strong Riemannian manifold. Hence we can consider the Banach manifold $\AC^p (I,G)$ for every Hilbert Lie group $G$.
 Then we obtain the following (see Proposition \ref{prop: HB:Lie}).
 \smallskip
 
 \textbf{Proposition B} \emph{Let $G$ be a Hilbert Lie group. Then for every $p \in \N$ the manifold structure of Theorem A turns $\AC^p (I,G)$ with the pointwise operations into a Banach Lie group. This Lie group structure coincides with the one constructed in \cite{hgmeasure16}.}\smallskip
 
 After establishing the general theory for manifolds of absolutely continuous functions, we investigate in Section \ref{sect: SRVT} the square root velocit transform in this settinIn  we begin the investigation of an extension of the square root velocity transform. 
 To put the results into context, let us recall the classical case from \cite{srivastava11sao}: The SRVT framework for vector valued curves.
 
 Let $I = [0,1]$ and define the Preshape space $\mathcal{P} = \text{Imm} (I,\R^n)$ of all smooth immersions. As we wish to identify shapes up to translation we restrict ourselves to the submanifold $\mathcal{P}_* = \{c \in \mathcal{P} \mid f(0)=0\}$.
 To obtain unparametrised curves declare $c_1, c_2 \in \mathcal{P}_*$ to be equivalent, if there is a smooth strictly increasing diffeomorphism $\varphi$ of $I$ with $c_1 = c_2 \circ \varphi$.
 The shape space is then defined as $\mathcal{S} \coloneq \mathcal{P}_*/\sim$.
 One now obtains a metric on $\mathcal{P}_*$ which descends to $\mathcal{S}$.  In the SRVT framework, this distance is induced by the square root velocity transform 
 \begin{displaymath}
  \SRVT \colon \mathcal{P}_* \rightarrow C^\infty  (I,\R^{n} \setminus \{0\}) ,\quad  c \mapsto  \frac{\dot{c}}{\sqrt{\norm{\dot{c}}}}.
 \end{displaymath}
 Observe that the $\SRVT$ consists of two essential steps: Differentiation and Scaling.
 The scaling in the $\SRVT$ is essential to obtain a metric which descents to $\mathcal{S}$. Moreover, it can be shown that the $\SRVT$ is a diffeomorphism and we can pull back the $L^2$-distance to obtain the distance on $\mathcal{P}$:\footnote{The resulting metric is in many cases the geodesic distance of a certain first order Sobolev metric. We omit these details and concentrate on the metric $d$ as this is the object of interest on the completion.}  
 \begin{displaymath}
  d(a,c) \coloneq \norm{\SRVT(a) - \SRVT (c)}_{L^2} = \int_0^1 \norm{\SRVT (a) (r) - \SRVT (c)(r)}\dd r.
 \end{displaymath}
 Now the space $C^\infty (I,\R^n \setminus \{0\})$ (equivalently $\mathcal{P}$) is neither geodesically complete (geodesics will pass through $0$) nor complete as a metric space. Its completion is of course $L^2 (I,\R^n)$, the space of square integrable functions. Similarly the completion of $\mathcal{P}_*$ is the space $\AC^1 (I,\R^n)_* = \{c \in \AC^1(I,\R^n) \text{ with } c(0) =0\}$.
 Also the $\SRVT$ can be extended to the completions. However, due to the scaling, the extended SRVT is no longer a diffeomorphism but only a homeomorphism (see \cite{bruveris2015} for a detailed account).
 Hence we leave the realm of Riemannian geometry and shape analysis is then carried out in the setting of metric spaces. 
 
 In the second part of the paper we study the SRVT now in the setting of manifold valued curves. We are interested in extensions of the SRVT to the metric completions of the smooth shape spaces (which will be manifolds of absolutely continuous functions). Apart from extending the SRVT framework (and clarifying continuity and differentiability properties of the SRVT on these manifolds), we are interested in the question of whether the SRVT induces a homeomorphism for manifold valued spaces.
 In particular, we explain how this setting extends the results on Lie group valued shape spaces achieved in \cite{celledoni15sao}.

 \section{Manifolds of absolutely continuous curves}\label{sect: MFDabs}
 
 In this section we construct the manifold structure on absolutely continuous curves. Basic defintions and facts on $L^p$-spaces and vector valued absolutely continuous functions are recorded in Appendix \ref{App: abscont}.
 Our exposition here is inspired by \cite{hgmeasure16} and \cite{FlaKli72,MR1330918}. 
 We will assume that the reader is familiar with Riemannian manifolds and concepts such as covariant derivatives (see e.g.\ \cite{MR1330918} or \cite{MR1666820}).
 
 \begin{setup}[Conventions]
  In the following $I = [a,b]$ for $a<b$ will denote a non-degenerate and compact interval, $(E,\norm{\cdot})$ a Banach space and $p \in [1,\infty[$.
  Denote by $B_r^{\norm{\cdot}} (x)$ (or by $B_r (x)$ if it is clear which Banach space is meant) the norm ball of radius $r$ around $x$ in $E$.
  Let $M$ be a Banach manifold. We will always assume that the manifolds are Hausdorff manifolds.
 \end{setup}
 
 \begin{defn}
 Define the set $\AC^p ([a,b] ,M)$ of all continuous functions $\gamma \colon [a,b] \rightarrow M$ for which there is a partition $a=t_0 < t_1 < \cdots < t_n =b$ and charts $(U_i,\phi_i)$ of $M$ for $1 \leq i \leq n$ such that 
  \begin{displaymath}
   \gamma ([t_{i-1},t_i]) \subseteq U_i \text{ and } \phi_i \circ \gamma|_{[t_{i-1},t_{i}]} \in \AC^p([t_{i-1}, t_i], \phi_i (U_i))
  \end{displaymath}
  Once this condition is satisfied for one partition of $[a,b]$ and a family of charts, it is automatically satisfied for all suitable partitions and families of charts (cf.\ \cite[Lemma 3.21]{hgmeasure16}).
 \end{defn}
 
 Our aim is to generalise the following result due to W.P.A.\ Klingenberg and collaborators (see e.g.\ \cite{FlaKli72} and cf.\ \cite{MR1330918}): 
 \smallskip
 
 \textbf{Theorem (Flaschel/Klingenberg '72)}\emph{ If $M$ is a finite-dimensional Riemannian manifold, then $\AC^2 ([a,b] ,M)$ becomes a Hilbert manifold}.
 \smallskip
 
 We will generalise it to infinite-dimensional Riemannian manifolds modifying the classical arguments.
 Note however, that in \cite{FlaKli72,MR1330918} it is crucial to assume that $M$ is finite-dimensional since several tools used in the proof are in general only available on finite-dimensional manifolds.
 First we need some model spaces for this manifold which turn out to be spaces of sections with sufficient regularity.
 
   \begin{defn}\label{defn: asec}
  Fix a Banach vector bundle $\pi \colon B \rightarrow M$ and $\gamma \in \AC^p ([a,b],M)$. 
  Now $c \in \AC^p ([a,b],B)$ with $\pi \circ c = \gamma$ is called \emph{(absolutely continuous) section over $\gamma$}.
  
  We will write $\AC_\gamma^p ([a,b],B)$ for the \emph{set of all absolutely continuous sections} over $\gamma$.  
  \end{defn}
 The pointwise operations turn $\AC_\gamma^p ([a,b],B)$ into a vector space. 
 Before we topologize these spaces, we need additional structures on the base manifold.
 
 \begin{setup}[Conventions]
  Assume that $M$ is endowed with a a strong Riemannian metric $G$. 
  Recall that each strong Riemannian metric possess an associated metric spray and a Levi-Civita covariant derivative $\nabla$. 
  We refer to \cite[Chapter VIII]{MR1666820} for a detailed discussion of strong Riemannian metrics (just called Riemannian metrics in loc.cit.) and their properties.
 \end{setup}
 
 The idea is to use parallel transport $P^\gamma_{t_0,t_1} \colon T_{\gamma (t_0)} M \rightarrow T_{\gamma (t_1)} M$ along smooth curves $\gamma \colon I \rightarrow M$. 
 Following \cite[VIII. \S 3]{MR1666820} parallel transport is smooth and on each fibre an isometry.
 Hence 
 $$Q_{\gamma,p,1} \colon \AC_\gamma^p ([a,b],TM) \rightarrow \AC_\gamma^p ([a,b] , T_{c(a)} M) , X \mapsto (t\mapsto P_{t,a}^\gamma \circ X (t))$$
 makes sense and is a bijection which we use to turn the left hand side into a Banach spaces (see Appendix \ref{App: abscont} for details). 
 
 To simplify the notation, we will work from now on only over the intervall $I = [0,1]$.
 Before we can formulate the result on the manifold structure, we need the concept of a smooth local addition. 
 The idea is to construct charts for the manifold of mappings by specifying with a local addition how trivialisations of $M$ ``move smoothly'' along a curve in $\AC^p (I, M)$.
 
 \begin{defn}			\label{def: local addition}
 A smooth map $\A \colon TM \supseteq \Omega \rightarrow M$ defined on an open neighbourhood $\Omega$ of the zero-section in $TM$ is called \emph{local addition} if 
  \begin{enumerate}
   \item $(\pi_{TM},\A) \colon TM \supseteq \Omega \rightarrow M\times M$ induces a diffeomorphism onto an open neighbourhood of the diagonal in $M \times M$,
   \item $\A (0_x) = x, \ \forall x \in M$, where $0_x$ is zero-element in the fibre over $x$. 
  \end{enumerate}
\end{defn}

\begin{remark}
 Local additions exist at least on all strong Riemannian manifolds (i.e.\ one can use the Riemannian exponential map, which is known to be a local diffeomorphism in this case, see \cite{MR1471480}, compare \cite{MR1330918}).
 Further, (possibly infinite-dimensional) Lie groups possess a local addition (cf.\ \cite[42.4]{MR1471480}).
\end{remark}

\begin{remark}\label{rem: loc:add:fibre}
 Note that the definition of a local addition $\A \colon \Omega \rightarrow M$ on a manifold $M$ entails that for every $x \in M$ the map $\A_x \coloneq \A|_{T_x M \cap \Omega} \colon T_x M \cap \Omega \rightarrow M$ is a diffeomorphism onto its open image.
 This follows from the fact that $i_x \colon M \rightarrow \{x\} \times M$ is a diffeomorphism, whence $\text{Im} \A_x = i_x^{-1} \left( (\pi_{TM},\A)(\Omega)\right)$ is open and on the image we have $\A_x^{-1} = (\pi_{TM}, \A)^{-1} \circ i_x$ is smooth.
\end{remark}

\begin{setup}[Canonical charts for $\AC(I,M)$]\label{setup: charts}
 Fix a local addition $\A \colon TM \supseteq \Omega \rightarrow M$ on $M$.
 Consider $f \in C^\infty (I,M) \subseteq \AC^p(I,M)$. 
 Observe that $\AC_f^p (I, \Omega)$ is open in $\AC_f^p(I,TM)$ by \ref{top:sectsp}.
 Then define the subset 
  \begin{displaymath}
   U_f \coloneq \{g\in \AC^p(I,M) \mid (f(t),g(t)) \in (\pi_{TM},\A)(\Omega)), \text{ for all } 0 \leq t \leq 1\}
  \end{displaymath}
 of $\AC^p(I ,M)$ together with a map $\Phi_f \colon U_f \rightarrow \AC_f^p (I, TM) $ given by 
  \begin{displaymath}
   \Phi_f (\gamma) \coloneq (\pi_{TM} ,\A)^{-1} \circ  (f,\gamma). 
  \end{displaymath}
 This mapping is bijective with inverse $\Phi_f^{-1} = (\A)_* \colon \AC_f (I,\Omega) \rightarrow U_f,\ c \mapsto \A \circ c$.
\end{setup}
 
 \begin{thm}\label{thm: ACman}
  Let $(M,G)$ be a strong Riemannian manifold with local addition $\A \colon \Omega \rightarrow M$. Then for every $p\in [1,\infty[$, the atlas 
    \begin{displaymath}
    \left\{ \Phi_f \colon U_f \rightarrow \AC_f^p (I,TM), \gamma \mapsto (\pi_{TM} ,\A)^{-1} \circ (f,\gamma) \right\}_{f \in C^\infty (I,M)} \text{ turns }  \AC^p (I,M)
    \end{displaymath}
  into an infinite-dimensional Banach manifold modelled on spaces of absolutely continuous sections. 
  The topology of $\AC (I,M)$ carries the identification topology with respect to the atlas, which is finer than the compact-open topology, whence Hausdorff and $C^\infty (I,M)$ is a dense subset. 
  Moreover, the construction is independent of the choice of Riemannian structure and of the local addition. 
 \end{thm}
 
 \begin{proof}
  We proceed in several steps. Let us first establish that the sets $U_f$ indeed cover $\AC (I,M)$.
  \medskip
  
  \textbf{Step 1:} \emph{Every $\gamma \in \AC (I,M)$ is contained in $U_f$ for some $f \in C^\infty (I,M)$.}
  Consider $\gamma \in \AC^p (I,M)$ and choose a partition $0=t_0 < t_1 < t_2 < \cdots < t_N =1$, $N\in \N$ and manifold chart $\kappa_i \colon U_i \rightarrow V_i \subseteq E$ of $M$ with $\gamma([t_{i-1},t_i]) \subseteq U_i$.
  Due to compactness of $\gamma ([t_{i-1},t_i])$ and Lemma \ref{lem: normcomp}, we can fix some $\varepsilon >0$ and a neighborhood $W_\gamma$ of $\gamma (I)$ such that for every $x \in W_\gamma$ we have $B^{G_x}_\varepsilon (0) \subseteq T_xM \cap \Omega$. 
  Here and in the following, $B^{G_x}_\varepsilon (0)$ will denote the norm ball in $T_xM$ with respect to the norm induced by the Riemannian metric $G$.
  Applying Lemma \ref{lem: pieces} for every piece in the partition, we construct smooth functions $f_i \colon [t_{i-1} , t_i] \rightarrow W_\gamma$ with $f_i (t_{i-1}) = \gamma (t_{i-1}), f_i(t_i) = \gamma (t_i)$ and $\norm{(\pi_{TM}, \A)^{-1} (f_i(t),\gamma(t))}_{G_{f_i(t)}} < \varepsilon$ for all $t \in [t_{i-1},t_i]$.
  Using cut-off functions, we can glue the $f_i$ together and obtain a smooth function $f \colon I \rightarrow W_\gamma$ with $\norm{(\pi_{TM}, \A)^{-1} (f(t),\gamma(t))}_{G_{f(t)}} < \varepsilon, \forall t \in I$.
  Now as $f(I) \subseteq W_\gamma$, we have by construction of $r_f$ the relation $\varepsilon \leq r_f$, whence $\gamma \in U_f$.
  \medskip
  
  \textbf{Step 2:} \emph{For $\eta \in U_f$ there is an open $\eta$-neighborhood $W_\eta^f$ in the compact open topology with $W_\eta^f \subseteq U_f$.}
  Consider $\eta \in U_f$.   
  By compactness of $f(I)$ and $\eta(I)$ there is a finite partition $0=t_0 <t_1 < \ldots t_N =1$ of $I$ together with pairs of charts $(\kappa_i,U_i), (\varphi_i,V_i)$ such that for $I_i \coloneq [t_{i-1},t_i]$ we have $\eta (I_i) \subseteq U_i$ and $f(I_i) \subseteq V_i$.
  Apply now Lemma \ref{lem: sandwich} (with $C=\infty$, cf.\ the statement of the Lemma) to each intervall $I_i$ and the charts. Thus we obtain a constant $s>0$ such that the set 
  $$W_\eta^f \coloneq \left\{g\in \AC^p(I,M) \mid g(I_i) \subseteq U_i \text{ and } \sup_{t \in I_i}\norm{\kappa_i \circ g(t) - \kappa_i \circ \eta(t)} <s\right\}$$
  is contained in $U_f$. Clearly $W_\eta^f$ is open in the compact open topology. 
  \medskip
  
    \textbf{Step 3:} \emph{Change of charts are smooth.}
  Consider $f,g \in C^\infty (I,M)$ with $O := U_\alpha \cap U_\beta \neq \emptyset$. 
  Fix $\eta \in O$.
  Construct $W_\eta^g$ together with the constant $s$ as in Step 2. Refining the partition of $I$ constructed in Step 2, we may assume that also $f (I_i)$ is contained in some chart $(\theta_i,Y_i)$. Now apply Lemma \ref{lem: sandwich} to  the mappings $f, \eta$ for each $I_i$ and the pairs of charts $(\theta_i,Y_i)$ and $(\kappa_i,U_i)$ but set $C\coloneq s$. 
  We deduce that there is $r>0$ such that $\Phi_f^{-1} (B_r^\infty (\Phi_f (\eta))) \subseteq W_\eta^g$. Hence on the open set $B_r^\infty (\Phi_f (\eta))$ the composition $\Phi_g \circ \Phi_{f}^{-1}$ makes sense. 
  Note that this entails that $\Phi_f (O)$ is open (as it is a neighborhood of each of its elements).  
  
  We will now apply Proposition \ref{prop: ACOmega:Rep}. Note that $O \coloneq \bigcup_{t \in I} \{t\} \times B_r^{G_{f(t)}} (\Phi_f (\eta)(t))$ is an open subset of $f^* TM$.
  On $O$ we define a smooth fibre-preserving map via
  \begin{equation}\begin{aligned}
  \tau_{f}^g \colon f^*TM \supseteq O  \rightarrow g^*TM, \quad
					  (t,x)\mapsto (t, (\pi_{TM},\A)^{-1} (g(t),\A (x))).\end{aligned} \label{chchmap}
   \end{equation}
  We then have (up to harmeless identification) $\Phi_g \circ \Phi^{-1}_f = (\tau_f^g)_*$ on $B_r^\infty (\Phi_f (\eta))$. 
  Now $(\tau_f^g)_*$ is smooth by Proposition \ref{prop: ACOmega:Rep} and we conclude that $\Phi_g \circ \Phi^{-1}_f$ is smooth on $B_r^\infty (\Phi_f (\eta))$. 
  As $\eta$ was arbitrary, the change of charts is smooth and reversing the r\^{o}les of $f$ and $g$ the same holds for $\Phi_f \circ \Phi_g^{-1}$.
  \medskip
  
  \textbf{Step 4:} \emph{Properties of the manifold topology of $\AC (I,M)$.}
  Endow $\AC (I,M)$ with the final topology with respect to the parametrisations $\{\Phi_f^{-1} \colon \AC_f (I,\Omega) \rightarrow U_f\}_{f \in C^\infty (I,M)}$. 
  We have to establish the Hausdorff property for this topology. To this end, we argue that it is finer than the compact open topology (which clearly is Hausdorff).
  A typical (sub)-basic neighborhood of $\eta \in \AC^p (I,M)$ is of the form $N(h,K,U) \coloneq \{h \in \AC^p (I,M) \mid h(K) \subseteq U\}$ for some $K \subseteq I$ compact and $U\subseteq M$ open. Choose $f \in C^\infty (I,M)$ with $\eta \in U_f$. Adapting the construction from Step 2, we may assume that the partition of $I$ was chosen such that for $I_i \cap K \neq \emptyset$ the chart $(\kappa_i,U_i)$ satisfies $U_i \subseteq U$. The resulting set $W_\eta^f$ is open in the final topology (using Step 3) and satisfies $W_\eta^f \subseteq N(h,K,U)$,
  Since $\eta$ was arbitrary, we infer that the manifold topology is finer than the compact open topology.
  
  To see that $C^\infty (I,M)$ is a dense subse in the manifold topology, note that it suffices to prove the property locally, i.e. on the chart domains $U_f$ for $f \in C^\infty (I,M)$.
  However, since $\Phi_f^{-1} = \Sigma_* \colon \AC_f (I,TM) \rightarrow \AC (I,U_f)$ takes smooth functions to smooth functions, density of smooth functions follows from Remark \ref{rem:dense}.
  \medskip
  
  \textbf{Step 5:} \emph{The manifold structure is independent of the choice of the Riemannian metric and the choice of local addition.}
  This follows directly from Proposition \ref{prop: postcom:sm} by taking $\psi = \id_M$.
  \end{proof}

  \begin{proposition}\label{prop: postcom:sm}
    Let $\psi \colon M \rightarrow N$ be a smooth map between strong Riemannian manifolds $((M,G),\A_M)$ and $((N,H),\A_N)$ with local additions. 
    Then $$\psi_* \colon \AC^p (I,M) \rightarrow \AC^p(I,N),\quad \gamma \mapsto \psi\circ \gamma$$ is smooth.
  \end{proposition}

  \begin{proof}
   \textbf{Step 0:}\emph{ Continuity of $\psi_*$.} 
   The topology of $\AC^p (I,M)$ is the final topology with respect to the parametrisations $\{\Phi_f^{-1}\}_{f \in C^\infty (I,M)}$.
   Hence it suffices to prove that $\psi_* \circ \Phi^{-1}_f$ is smooth for each $f \in C^\infty (I,M)$ to establish the continuity of $\psi_*$.
   
   Notice however that the map $\psi_* \colon C(I,M) \rightarrow C(I,N)$ is continuous with respect to the compact-open topology on both spaces.
   \medskip
   
   \textbf{Step 1:} \emph{Reduction to special pairs of charts.}
   Let $\gamma \in \AC^p (I,M)$ and consider $\psi \circ \gamma$. 
   An easy modification of the proof of Lemma \ref{lem: pieces} shows that we can construct an open $\psi\circ \gamma$-neighborhood $O_{\psi\circ \gamma} \subseteq C (I,M)$ in the compact open topology such for every smooth $f\in O_{\psi\circ \gamma}$ we have $\psi \circ \gamma \in U_f$. Now $(\psi_*)^{-1}(O_{\psi\circ \gamma})$ is an open neighborhood of $\gamma$ in the compact open topology (by Step 0). Adapting the construction in Step 1 of the proof of Theorem \ref{thm: ACman} we see that we can obtain a smooth map $c \colon I \rightarrow M$ with $\gamma \in U_c$ and $c \in \psi_*^{-1} (O_{\psi \circ \gamma})$.
   Thus $\psi \circ \gamma \in U_{\psi \circ c}$ by definition of $O_{\psi\circ \gamma}$.
   
   Now Step 2 of the proof of Theorem \ref{thm: ACman} asserts that there is a compact open $\gamma $-neighborhood $W_{\psi \circ \gamma}^{\psi\circ c} \subseteq U_{\psi\circ c}$. By Step 0, we see that $O \coloneq \psi_*^{-1}(W_{\psi \circ \gamma}^{\psi\circ c} \cap W_\gamma^c \subseteq U_c$ is an open neighborhood of $\gamma$ on which $\Phi_{\psi \circ c} \circ \psi_* \circ \Phi^{-1}_{c}$ makes sense. Hence it suffices to check that the map is smooth on all such sets for all $c \in C^\infty (I,M)$. 
   \medskip
   
   \textbf{Step 2:} \emph{Smoothness in special pairs of charts.} Let $\gamma \in \AC^p (I,M)$ and consider $c \in C^\infty (I,M)$ and $O$ as in Step 1.
    Let us show that $\Phi_{\psi \circ c} \circ \psi_* \circ \Phi_c^{-1}|_O \colon \AC_c^p (I,TM) \supseteq O \rightarrow \AC_{\psi\circ c} (I,TM)$ is smooth.
   Let $h \in O$ and identify the sections over $c$ and $\psi \circ c$ with sections in the respective pullback bundles. Since $O$ is open in the compact open topology, there is $R>0$ such that the following map makes sense (up to identification) for each $x \in B_R^{F_{c(t)}} (0)$ and $t\in I$:
   \begin{displaymath}
   F_h (t,x)\coloneq (t,(\pi_{TN}, \A_N)^{-1} \circ (\psi \circ c, \psi) (t, \A_M (x+h(t)) 
   \end{displaymath}
   By construction $F_h$ is a smooth fibre-preserving map, whence $(F_h)_*$ is smooth by Proposition \ref{prop: ACOmega:Rep}.
   A quick calculation shows that $\Phi_{\psi \circ c} \circ \psi_* \circ \Phi_c^{-1}(\delta) = (F_h)_* (\delta -h)$ holds for $\delta \in B_R^\infty (h)$.
   As $h$ was arbitrary, this proves that $\Phi_{\psi \circ c} \circ \psi_* \circ \Phi_c^{-1}|_O$ is smooth.
   \end{proof} 
   
   \begin{remark}[]
    For $((M,g)\A_M)$ and $((N,h),\A_N)$ strong Riemannian manifolds with local addition, $\A_M \times \A_N$ is a local addition for the product manifold $M\times N$.
    Hence the usual arguments involving the product structure show that as manifolds $\AC^p(I,M\times N) \cong \AC^p(I,M) \times \AC^p(I,N)$ 
   \end{remark}
 
 \begin{remark}
  For $p=2$ and $M$ finite-dimensional Theorem \ref{thm: ACman} and Proposition \ref{prop: postcom:sm} are classical results due to Klingenberg and Flaschel.
  We refer to \cite{FlaKli72} for more details (or \cite{MR1330918} for a shorter account).
 \end{remark}
  The next results allow us to compare the manifold of absolutely continuous curves to other manifolds which arise from curves with values in $M$.
 \begin{proposition}\label{prop: cinf:AC}
  Let $(M,g)$ be a strong Riemannian manifold and consider $C^\infty (I,M)$ as a smooth manifold with the compact open $C^\infty$-topology (cf.\ \cite[Theorem 10.4]{MR583436}).
  Then the inclusion $\iota_{C^\infty} \colon C^\infty (I,M) \rightarrow \AC^p (I,M)$ is smooth.
  \end{proposition}
  
  \begin{proof}
   Recall from \cite[Theorem 10.4]{MR583436} that an atlas for $C^\infty (I,M)$ is given by the maps $\{\Phi_f^{\infty} \colon C^\infty (I,M) \supseteq U_c^\infty \rightarrow C^\infty_f (I,TM), g \mapsto (\pi_{TM},\A)^{-1} \circ (f,g)\}_{f\in C^\infty (I,M)}$, where $C^\infty_c (I,TM)$ denotes the subset of smooth curves in $TM$ over $c$. 
   Following \cite[Remark 42.2]{MR1471480}, the compact-open $C^\infty$-topology is the final topology with respect to the family $\{(\Phi_f^{\infty})^{-1}\}_{f \in C^\infty (I,M)}$ since $I$ is compact.
   As $U_f^\infty = U_f \cap C^\infty (I,M)$ we only need to establish smoothness of $\iota_{C^\infty}$ in the pairs of charts $(\Phi_f^{\infty},\Phi_f)$ for each $f \in C^\infty (I,M)$.
   In these charts, $\iota_{C^\infty}$ is identified with the inclusion $C^\infty_f (I,TM) \rightarrow \AC_f^p (I,TM)$.
   
   Combining \cite[Lemma 1.6.2]{MR1330918} with \cite[III. Proposition 1.3]{MR1666820} parallel transport induces a bundle isomorphism $P^f \colon f^*TM \rightarrow I \times T_{f(0)}M, (t,v) \mapsto (t,P_{0,t}^f)^{-1} (v))$. 
   By Gl\"{o}ckners $\Omega$-Lemma \cite[Corollary F.24]{hg2004c}, the map $(\text{pr}_2 \circ P^f)_*\colon C^\infty_f (I,TM) \rightarrow C^\infty (I,T_{f(0)})$ is a diffeomorphism.  
   Now the diffeomorphsims $(\text{pr}_2 \circ P^f)_*$ and $Q_{c,p,1}$ take $C^\infty_f (I,TM) \rightarrow \AC_f^p (I,TM)$ to the canonical inclusion $C^\infty (I,T_{c(0)}M) \rightarrow \AC^p (I,T_{c(0)})$.
   This inclusion is continuous linear (whence smooth) as \ref{rem: ISO} implies together with the continuity of $C^\infty (I,T_{c(0)}M) \rightarrow C^0 (I,T_{c(0)}M) \rightarrow L^2(I,T_{c(0)}M), \ f\mapsto \dot{f}$ (use \cite[Proposition 2.2]{hgmeasure16} and continuity of the derivation operator, cf.\ \cite{MR3342623}).
  \end{proof} 
 
  As in the finite-dimensional case, (see \cite{FlaKli72}), the Riemannian manifold $(M,G)$ can be identified with an embedded closed submanifold of $\AC^p (I,M)$.
   
  \begin{lem}\label{lem: auxmap}
  	Let $(M,G)$ be a strong Riemannian manifold. Then
  	\begin{enumerate}
  		\item for each $t \in I$ the map $\text{ev}_t \colon \AC^p (I,M) \rightarrow M, c \mapsto c(t)$ is a smooth submersion.
  		\item The mapping \begin{displaymath}
  		\iota_M \colon M \rightarrow \AC^p (I,M) ,\quad  m \mapsto (t \mapsto m)
  		\end{displaymath}
  		is a smooth embedding. Thus $M$ becomes a closed submanifold of $\AC (I,M)$.
  	\end{enumerate}
  \end{lem}
  
  \begin{proof}
  	\begin{enumerate}
  		\item Since $\AC^p (I,M)$ carries the final topology with respect to the canonical charts, $\ev_t$ will be continuous if $\ev_t \circ \Phi^{-1}_c$ is smooth for each $c \in C^\infty (I,M)$.
  		Now the diffeomorphism $\Phi^{-1}_c \circ Q_{c,p,1}^{-1}$ identifies $\ev_t$ with $\tilde{\ev}_t \colon \AC^p (I,T_{c(0)}M) \rightarrow T_{c(0)}M, h \mapsto h(t)$.
  		By Remark \ref{rem: ISO} $\tilde{\ev}_t$ is continuous linear, whence smooth.
  		In particular, its derivative $d\tilde{\ev}_t (\eta,\xi) = \tilde{\ev}_t (\xi)$. Thus $T_{c(0)} M \rightarrow T_{c(0)} M \times L^p (I,T_{c(0)}) \cong \AC^p (I,T_{c(0)}M),\ m \mapsto (m,0)$ is a continuous linear section of $d\tilde{\ev}_t$. Hence $\tilde{\ev}_t$ is a submersion by \cite[Theorem A]{glockner15fos} and we deduce that $\ev_t$ is a smooth submersion.
  		\item  Let $m\in M$ then on the open neighborhood $W_m =  \A(T_mM \cap \Omega)$ of $m$ the map $\Phi_{\iota_M (m)} \circ \iota_M$ makes sense (as $U_{\phi_m}=\{\gamma \in \AC^p (I,M)\mid \gamma (I) \subseteq \A(T_mM \cap \Omega)\}$). After an identification we see that $\Phi_{\iota_M (M)}\iota_M (x) = (\pi_{TM},\A)^{-1} (m,x), 0) \in T_{m} M \times L^p (I, T_m M)$ is clearly smooth. As $\AC^p (I,M)$ carries the identification topology (and $\iota_M (M)$ is covererd by charts of the type $\Phi_{\iota_M (m)}$), this already implies that $\iota_M$ is a smooth map.
  		Further, its derivative (again read off in these charts) is easily seen to be embedding of $T_mM$ onto a complemented subspace. Thus \cite[Theorem H]{glockner15fos} shows that $\iota_M$ is an immersion.
  		
  		Since $\ev_0 (\iota_M (m)) = m$ an inverse for $\iota_M$ is the continuous map $\ev_0|_{\iota_M (M)}$. We deduce that $\iota_M$ is a topological embedding which is also a smooth immersion. Summing up, $\iota_M$ is a smooth embedding. Its image is closed since $\iota_M (M) = \bigcap_{t \in I} \{c \in \AC^p (I,M) \mid \ev_0 (c) = \ev_t (c)\}$.\qedhere
  	\end{enumerate}
  \end{proof}
 
 \section{\texorpdfstring{The tangent and the $L^p$-bundle over $\AC^p (I,M)$}{The tangent and the Lp-bundle}} 
 
 In this section we discuss two canonical bundles over $\AC^p(I,M)$. Again $(M,G)$ will throughout be a strong Riemannian manifold.. 
 First, we identify the tangent manifold $T\AC^p (I,M)$ as a bundle of absolutely continuous functions with values in the tangent manifold.
 As a consequence, the functor $\AC^p (I,\cdot)$ commutes with the tangent functor. 
 In a second step, we then investigate a bundle of $L^p$-functions over $\AC^p(I,M)$.
 We will then see that the derivative $\gamma \mapsto \dot{\gamma}$ induces a differentiale section of the $L^p$-bundle.
 Again for the classical case (i.e.\ $p=2$ and $\dim M < \infty$) these results are due to Klingenberg and Flaschel \cite{FlaKli72,MR1330918}.
 
 To identify the tangent manifold recall first that the tangent manifold of a strong Riemannian manifold admits a canonical Riemannian metric (cf.\ \cite[X \S 4]{MR1666820}).
 
 \begin{setup}[Sasaki metric]
  Let $(M,G)$ be a strong Riemannian manifold and denote its associated connection (see e.g.\ \cite[X \S 4 Lemma 4.4]{MR1666820}) by $K \colon T^2M \rightarrow TM$. 
  Then $T_\xi TM$ splits into a direct sum of the vertical space $T_{\xi v} TM \coloneq \ker T\pi_{TM}|_{T_\xi TM}$ and the horizontal space $T_{\xi h} TM \coloneq \ker K|_{T_\xi TM}$.
  Combining these spaces, we obtain a splitting of $T^2M  = \hor (T^2M) \oplus \ver(T^2M)$ into the horizontal and the vertical subbundle such that 
     \begin{align*}
      T_\xi \pi_{TM} \colon \hor (T^2M)_\xi \rightarrow T_{\pi_{TM}(\xi)} M \text{ and }
      K|_{\ver(T^2M)_\xi} \colon \ver (T^2M)_X \rightarrow T_{\pi_{TM} (X)} M
     \end{align*}
   are isomorphisms (cf.\ \cite[1.5.10 Proposition]{MR1330918}).
   This induces a strong Riemannian metric $G_{Sa}$ on $TM$ by declaring for $\xi \in TM$ that
     \begin{align*}
     G_{Sa, \xi} (v,w) = G_{\pi_{TM}(\xi)}(T_\xi \pi_{TM} (v),T_\xi \pi_{TM}(w)) + G_{\pi_{TM}(\xi)}(K(v),K(w)).
     \end{align*}
  We call this metric the \emph{Sasaki Metric} on $TM$. By definition the horizontal and vertical subspace are orthogonal.
 \end{setup}
 
 Using the Sasaki metric, we obtain for every strong Riemannian manifold $(M,G)$ a Banach manifold $\AC^p (I,TM)$ such that $(\pi_{TM})_* \colon \AC^p (I,TM) \rightarrow \AC^p (I,M)$ is smooth.
 We will now explain how $\AC^p (I,TM)$ becomes a vector bundle over $\AC^p(I,M)$.
 
 \begin{setup}\label{setup: VBcharts}
 Let us denote from now on by $\exp \colon \Omega \rightarrow M$ the Riemannian exponential map of the strong Riemannian manifold $(M,G)$.
 Further, we let $\Omega \subseteq TM$ be chosen such that $\exp$ restricts on $\Omega \cap T_xM$ to a diffeomorphism onto an open set for each $x \in M$.
 Define the map 
 \begin{displaymath}
  \tau \colon TM \oplus TM \supseteq \Omega \times TM \rightarrow TM, \quad (\xi_x, v_x) \mapsto T_{\xi_x} \exp \circ (K|_{T_{\xi_x v}TM})^{-1} (v_x) \quad \text{ for } \xi_x,v_x \in T_xM.
 \end{displaymath}
  Note that as a consequence of the Dombrowski splitting theorem \cite[X, \S 4 Theorem 4.5]{MR1666820} we have $TM \oplus TM \cong \ver (T^2M)$ and this isomorphim is fibre-wise given by $(K|_{T_{\xi_x v}TM})^{-1} (v_x)$.
  Hence $\tau$ is smooth. 
 Observe that for each $\xi_x \in \Omega$ we obtain a linear isomorphism 
 $\tau_{\xi_x} \colon T_{x} M \rightarrow T_{\exp (\xi)} M, v_x \mapsto T_{\xi_x} \exp \circ (K|_{T_{\xi_x v}TM})^{-1} (v_x)$.
 By construction the map $\tau_{\xi_x}$ is the restriction of $T_\xi \exp$ to elements which are only non-trivial in the fibre-direction in $TM$.
 \end{setup}

  \begin{lem}\label{lem: VB:charts}
   Let $c \in C^\infty (I,M)$, then the following map is a smooth diffeomorphism 
    \begin{align*}
     \Psi_c^{-1} \colon \AC^p_c(I, \Omega) \times \AC^p_c (I, TM) &\rightarrow (\pi_{TM})_*^{-1} (U_c) \subseteq \AC^p (I,TM),\\ (\xi, \eta) &\mapsto \left( t \mapsto \tau \circ (\xi(t),\eta(t))\right).
    \end{align*}
  Further, $\{\Phi_c\}_{c\in C^\infty (I,M)}$ is a bundle atlas turning $(\pi_{TM})_* \colon \AC^p (I,TM) \rightarrow \AC^p (I,M)$ into a Banach bundle.
  \end{lem}

  \begin{proof}
  Taking identifications, $\Psi_c^{-1}$ corresponds to $\tau_* \colon \AC^p_c (I, \Omega \oplus TM) \rightarrow \AC^p(I,TM)$ by \ref{setup: Whitney}.
  Before we prove that $\tau_*$ is smooth, let us discuss charts for $\AC^p (I,TM)$.
  
  As the construction of $\AC^p (I,M)$ did not depend on the choice of local addition, we choose on $M$ the Riemannian exponential map $\exp \colon \Omega \rightarrow M$ as our local addition.
  Note that $T\exp$ is \textbf{not} a local addition on $TM$, since $T\Omega$ is not a neighborhood of the zero section.
  To remedy this, consider the canonical flip $\chi \colon T^2 M \rightarrow T^2 M$ of the double tangent bundle (i.e.\ the unique bundle isomorphism which in local charts is given by $(x,y, a,b) \mapsto (x,a,y,b)$).
  Then $T\exp \circ \chi \colon \chi (T\Omega) \rightarrow TM$ is a local addition for $TM$ (see \cite[Lemma 7.5]{MR3351079}) and we construct the canonical charts $\{\Phi_f\}_{f\in C^\infty (I,TM)}$ on $\AC^p (I,TM)$ with respect to $T\exp \circ \chi$.
  let $0_c \in C^\infty (I,TM)$ be the zero-section over $c$. Then 
  \begin{displaymath}
   \tau (\xi_x, v_x) = T_{\xi_x} \exp  T_{\xi_x} \exp \circ (K|_{T_{\xi_x v}TM})^{-1} (v_x) =  T_{\xi_x} \exp \circ (\chi \circ \chi) \circ (K|_{T_{\xi_x v}TM})^{-1} (v_x).
  \end{displaymath}
  We now obtain a smooth fibre-preserving map 
  $$\hat{\tau} \colon c^* (\Omega \oplus TM) \rightarrow 0_c^* (\chi (T\Omega)) \subseteq 0_c^* (T^2M), (\xi_x,v_x) \mapsto \chi \circ (K|_{T_{\xi_x v}TM})^{-1} (v_x)$$
  which is a diffeomorphism onto its image (as $\chi$ and the identification are invertible).
  Now $\tau_* = \Phi_{0_c}^{-1} \circ (\hat{\tau})_*$ is bijective. Since $\Phi_{0_c}$ is a chart for $\AC^p(I,TM)$ and $(\hat{\tau})_*$ is smooth by Remark \ref{rem: ext:ACProP}, the map $\Psi_c^{-1}$ is a smooth diffeomorphism.
  
  To see that we obtain a bundle atlas, note first that $(\pi_{TM})_* \circ  \Psi_c^{-1} = \Phi_c^{-1}$. Hence the pairs of charts are compatible.
  For $c,d \in C^\infty (I,M)$ we obtain the formula 
    \begin{equation}\label{bundle trans}
     \Psi_c \circ \Psi_d^{-1} = (\Phi_c \circ \Phi_d^{-1} , (T_{\text{fib}}\tau_d^c)_*) \colon \Phi_d^{-1} (U_c \cap U_d) \times \AC^p_d (I,TM) \rightarrow \AC^p_c (I,TM),
    \end{equation}
    where $T_{\text{fib}}\tau_d^c$ is the fibre-derivative of the fibre-preserving map $\tau_c^d$ from \eqref{chchmap}.
  In particular, for each fixed $\eta  \in \Phi_d^{-1} (U_c \cap U_d)$, the map $\Psi_c \circ \Psi_d^{-1} (\eta , \cdot)$ is a linear isomorphism since $T_{\text{fib}}\tau_d^c (\eta(t),\cdot)$ is linear. 
  We conclude that the transition maps of the atlas are continuous isomorphisms on each fibre, whence the mappings $\{\Psi_c \mid c \in C^\infty(I,M) \}$ form a bundle atlas. 
  \end{proof}
 
 Observe that the transition maps of the bundle atlas $\{\Psi_c\}$ coincide with the transition maps of the tangent atlas for $T\AC^p (I,M)$ by \eqref{bundle trans}.
 This will allow us to identify the tangent bundle of $\AC^p (I,M)$.
 However, let us recall from \cite[II. Lemma 3.10]{FlaKli72} the following construction: 
 
 \begin{setup}[Canonical chart centered around an arbitrary maps in $\AC^p (I,M)$]\label{setup: mcharts}
  Since $\exp$ is a local addition, $(\pi_{TM},\exp) \colon \Omega \rightarrow V \times V \subseteq M \times M$ is a diffeomorphism with an open image. 
  Thus Proposition \ref{prop: postcom:sm} shows that $(\pi_{TM},\exp)_* \colon \AC^p (I,\Omega) \rightarrow \AC^p (I,V\times V) \cong \AC^p (I,V) \times \AC^p (I,V)$ is again a diffeomorphism with open image. 
  Arguing as in Remark \ref{rem: loc:add:fibre}, we see that $\Phi_\gamma^{-1} \colon \AC^p_\gamma (I,\Omega) \rightarrow \exp_* (\AC^p_\gamma (I,\Omega)) \subseteq \AC^p (I,M), \eta \mapsto \exp \circ \eta$ is a diffeomorphism with open image.
  Hence we obtain canonical charts centered at $\gamma$ for each $\gamma \in \AC^p (I,M)$ (note that by construction the inverse of $\Phi_\gamma^{-1}$ is a map $\Phi_\gamma$ as constructed in \ref{setup: charts}).
 \end{setup}
 
  Let us denote elements in the tangent bundle $T\AC^p (I,M)$ by equivalence classes $[\Phi_c,X,Y]$ where $\Phi_c$ is a canonical chart, $\Phi_c^{-1} (X)$ is the base point of the tangent space and $Y \in \AC_c^p (I,M)$ (and the equivalence relation is $(\Psi_c, X,Y) \sim (\Psi_g, A,B)$ iff $T(\Psi_c \Psi_g^{-1}) (A,B) = (X,Y)$). 
  With these notations and our preparation, the proof of \cite[II. Satz 3.11]{FlaKli72} carries over almost verbatim to the following result.  
    
    \begin{proposition}
    Let $(M,G)$ be a strong Riemannian manifold.    
    Then $$i_{TM} \colon T \AC^p (I,M) \rightarrow \AC^p (I,TM),\ T_\gamma \AC^p_c (I,M) \ni [\Phi_\gamma ,0_\gamma, X] \mapsto X \in \AC_\gamma^p (I,TM)$$ is a bundle isomorphism.
    Moreover, if $(N,H)$ is a strong Riemannian manifold and $\psi \colon M \rightarrow N$ smooth, the following diagram commutes.
     \begin{displaymath} 
    \begin{xy}
  \xymatrix{
      T\AC^p (I,M) \ar[rd]^{i_{TM}} \ar[dd]_{\pi_{T\AC (I,M)}} \ar[rrrr]^{T(\psi_*)}&&&&  T\AC^p(I,N) \ar[ld]_{i_{TN}} \ar[dd]^{\pi_{T\AC (I,N)}}\\
       & \AC^p (I,TM) \ar[rr]^{(T\psi)_*} \ar[d]^{(\pi_{TM})_*} &  & \AC^p (I,N)  \ar[d]^{(\pi_{TM})_*}  & \\
      \AC^p (I,M) \ar@{=}[r]&  \AC (I,M)  \ar[rr]^{\psi_*} & & \AC^p(I,N) \ar@{=}[r]  &\AC^p(I,N)
  }
\end{xy}   
     \end{displaymath}
  \end{proposition}
  
  \begin{proof} 
   By definition of the tangent bundle $i_{TM}$ is a linear isomorphism of each fibre. Hence it suffices to check smoothness in local charts, i.e.\ for smooth $c \in C^\infty (I,M)$ we have to check that $\Psi_c \circ i_{TM} \circ T\Phi_c^{-1} \colon \AC^p_c (I,\Omega) \times \AC^p_c(I,TM) \rightarrow \AC^p_c (I,\Omega) \times \AC^p_c(I,TM)$ is smooth. To see this one argues in a neighborhood of $\{d\} \times \AC^p_c(I,TM) \in \AC^p_c (I,\Omega) \times \AC^p_c(I,TM)$ and uses Theorem \ref{prop: ACOmega:Rep} to establish smoothness of the resulting fibre-preserving maps. 
   We omit the technical computation here, since it is carried out in the proof of \cite[II. Satz 3.11]{FlaKli72} in great detail. 
   Copying the argument (note that the charts $\text{Exp}_{0_c}^{-1}$ and $T\exp_c$ as in loc.cit.\ are in our notation $\Psi_c$ and $ T\Phi_c^{-1}$) from \cite[II. Satz 3.11]{FlaKli72} together with the formula \eqref{bundle trans} yields the desired results.
   
   For the readers convenience we copy the computation from \cite[II. Satz 3.11 (ii)]{FlaKli72} to obtain the desired formula for $T(\psi_*)$.
   Consider $X \in \AC^p_c (I,M)$, then
    \begin{align*}
     i_{TN} \circ T(\psi_*) \circ (i_{TM})^{-1} (X) &\stackrel{\hphantom{\text{Proposition \ref{prop: ACOmega:Rep}}}}{=} d(\Phi_{\psi\circ c } \circ \psi_* \circ \Phi_c^{-1})(0_c ; X) \\
     &\stackrel{\text{Proposition \ref{prop: ACOmega:Rep}}}{=} (T_{\text{fib}} (\pi_{TN},\exp_N)^{-1} \circ (f\circ c, f \circ \exp_M))_* (0_c,X)\\
     &\stackrel{\hphantom{\text{Proposition \ref{prop: ACOmega:Rep}}}}{=}  (Tf)_* X   
    \end{align*}
    where $\exp_M$ (resp.\ $\exp_N$) is the Riemannian exponential map on $M$ (or $N$, resp.) and we have used for the last identification that the derivative of the Riemannian exponential map at the zero section is the identity, i.e.\ $T_{0_x}\exp (y) = y$. 
  \end{proof}

 In the next chapter, we want to consider the square root velocity transform. 
 Since one takes derivatives of curves in this transform, we are interested in the differentiability properties of this operations.
 By definition of absolutely continuous curves, the derivative of an absolutely continuous curve will (in general) not be absolutely continuous. 
 Instead we obtain an $L^p$-lifts of the curve to the tangent bundle. 
 
 \begin{defn}\label{defn:L2sect}
   Let $\pi_B \colon B \rightarrow M$ be a a smooth Hilbert bundle (see e.g.\ \cite[VII, \S 3.]{MR1666820}) and consider $\gamma \in \AC^p (I,M)$. A curve $c \colon I \rightarrow B$ with $\pi_B \circ c = \gamma$ such that for every vector bundle trivialisations $(\kappa, U)$ of $\pi \colon B \rightarrow M$ the map $\text{pr}_2 \circ \kappa \circ c$ is locally of class $L^p$ (i.e.\ for all closed subintervals such that the mapping is defined, its restriction is an $L^p$-function), is called \emph{$L^p$-section} over $\gamma$.
   We write $L^p_\gamma (I,M\leftarrow B)$ for the set of all these mappings.  
  \end{defn}
 
 Note that one can \textbf{not} define $L^p$-functions with values in $M$ in a similar way (to remind the reader of this we chose a somewhat cumbersome notation).
 Again, the pointwise operations turn $L^p_\gamma (I,M\leftarrow B)$ into a Banach space (we shall only need $B = TM$ and $B = TM \oplus TM$ and construct the topology in Appendix \ref{App: abscont}).
 Now taking derivatives of curves in $\AC^p(I,M)$ produces $L^p$-sections over these curves. 
 Moreover, the derivative induces a smooth map into a bundle of $L^p$-functons.
 
 \begin{proposition}
  The set $L^p (I, M\leftarrow TM) \coloneq \bigcup_{\gamma \in \AC^p (I,M)} L^p_\gamma (I,M \leftarrow TM)$ forms a smooth vector bundle $(\pi_{TM})_* \colon L^p (I, M\leftarrow TM) \rightarrow \AC^p(I,M)$ whose trivilasitions are given by 
  \begin{align*}
     \Psi_{c,0}^{-1} \colon \AC^p_c(I, \Omega) \times L^p_c (I,M\leftarrow TM) &\rightarrow (\pi_{TM})_*^{-1} (U_c) \subseteq L^p (I,M \leftarrow TM),\\ (\xi, \eta) &\mapsto \left( t \mapsto \tau \circ (\xi(t),\eta(t))\right)
    \end{align*}
  for $c \in C^\infty (I,M)$. Thus with respect to the final topology induced by $\{\Psi_{c,0}^{-1}\}_{c\in C^\infty(I,M)}$ the set $L^p (I, M\leftarrow TM)$ becomes a Hausdorff Banach manifold. 
   \end{proposition}
 \begin{proof}
 	By definition of $L^p (I, M\leftarrow TM)$, we have to test in bundle trivialisation if $\Psi_{c,0}^{-1} (\eta,\xi)$ is an $L^p$-section. 
    Hence we partition $I$ into smaller intervalls such that the image of $\Psi_{c,0}^{-1} (\eta,\xi)|_{I_i}$ is contained in the domain of $T\kappa_i$. To simplify the notation let us drop the subscripts in this argument. 
    We consider the smooth map $\gamma \colon I \times T_{c(0)}M \rightarrow E, \gamma (t,y) \coloneq \pr_2 \circ T \kappa (\tau (\eta(t),P^c_{0,t} (y))$. By construction $\gamma (t,\cdot)$ is linear for each $t \in I$.
    Now Lemma \ref{lem: ted:ext} shows that there is an open neighborhood $I \subseteq W$ and a smooth extension $\Gamma \colon U \times  T_{c(0)}M \rightarrow E$ for which $\Gamma(s,\cdot)$ is linear for each $s \in U$. 
    We have thus established the prerequesits of \cite[Lemma 2.1]{hgmeasure16} which allows us to deduce that $T\kappa \Psi_{c,0}^{-1} (\eta,\xi) = \Gamma (\text{inc}_I^W, Q_{c,p,0}^{-1} (\xi))$ is an $L^p$-function, where $\text{inc}_I^W \colon I \rightarrow W$ is the inclusion.
    Hence the definition of $\Psi_{c,0}^{-1}$ makes sense and the images of the maps $\Psi_{c,0}^{-1}$ clearly cover $L^p (I, M\leftarrow TM)$ with $(\pi_{TM})_* \Psi_{c,0}^{-1} = \Phi_c^{-1}$.
 	
 	We want $L^p (I, M\leftarrow TM)$ to become the total space of a bundle. To this end we only have to check the smoothness of the transition maps. 
 	Then \cite[III, \S 1 Proposition 1.2]{MR1666820} shows that we obtain a manifold with the desired topology. This manifold is Hausdorff since the base $\AC^p (I,M)$ is so.
 	Let us prove that the transition map $\Psi_{d,0} \circ \Psi_{c,0}^{-1} \colon \AC^p (I,O_{c}^d) \times  L^p_c (I,M\leftarrow TM) \rightarrow \AC_d^p (I,TM) \times L^p_d (I,M\leftarrow TM)$ is smooth for $c,d \in C^\infty (I,M)$, where $O_c^d$ is a the (open) subset on which the change of charts $\Phi_d \circ \Phi_c^{-1}$ is defined. As seen in the proof of Lemma \ref{lem: VB:charts}, the transition map satisfies $\Psi_{d,0} \circ \Psi_{c,0}^{-1} = (\Phi_{d}\circ \Phi_c^{-1}, (T_{\text{fib}}\tau_d^c)_*)$, where
 	\begin{displaymath}
 	 (T_{\text{fib}}\tau_d^c)_* \colon \AC^p_c (I,O_c^d) \times L^p_d (I,M \leftarrow TM), \ (\xi,\eta) \mapsto \tau \circ (\xi, \eta),
 	\end{displaymath}
 	and $\tau_c^d$ is the smooth map from \eqref{chchmap}. Since the change of charts 
 	is smooth, it suffices to prove that $(T_{\text{fib}}\tau_d^c)_*$ is smooth.
 	Pick $\xi \in \AC^p_c (I,O_c^d)$ and let us prove that $(T_{\text{fib}}\tau_d^c)_*$ is smooth on a neighborhood of $\{\eta\} \times  L^p_d (I,M \leftarrow TM)$. Fix $r0$ such that $B^\infty_r (\eta) \subseteq O_c^d$. Since addition in the Banach space $\AC^p_c (I,TM)$ is smooth, we may assume (by considering the mapping $(\gamma, \eta) \mapsto  (T_{\text{fib}}\tau_d^c)_* (\gamma +\eta, \xi)$ that $\eta =0$. Composing with the Banach space isomorphisms from \ref{top:sectsp} the transition will be smooth if \begin{displaymath}
Q_{d,p,0}^{-1} \circ (T_{\text{fib}}\tau_d^c)_* \circ (Q_{c,p,1} \times Q_{c,p,0}) \colon \AC^p (I, B_r^{G_{c(0)}} (0)) \times L^p(I,T_{c(0)}M) \rightarrow L^p(I,T_{d(0)}M)
 	\end{displaymath}
 is smooth.
 Consider now the smooth map 
 \begin{align*}
 \alpha \colon I \times B_r^{G_{c(0)}}(0) \times T_{c(0)}M &\rightarrow T_{d(0)}M,\\
 	(t,x,y) &\mapsto (P_{0,t}^d)^{-1}  (x) T_{\text{fib}} \tau_c^d (P_{0,t}^c  (x),P^c_{0,t}(y))									 			
 	\end{align*}	
 	and note that for all fixed pairs $(t,x)$ the map $\alpha(t,x,\cdot)$ is linear.
 We use again Lemma \ref{lem: ted:ext} to construct an open neighborhood $I \subseteq W$ (with $\text{inc}_I^W$ the inclusion) and a smooth extension $A \colon W  \times B_r^{G_{c(0)}}(0) \times T_{c(0)}M \rightarrow T_{d(0)}M$. Again, $\overline{\alpha}(s,x,\cdot)$ is linear for each $(s,x) \in W \times B_r^{G_{c(0)}}(0)$. Hence \cite[Proposition 2.2]{hgmeasure16} shows that $\tilde{A} \colon C(I, W  \times B_r^{G_{c(0)}}(0)) \times L^p (I,T_{c(0)}M) \rightarrow L^p (I,T_{d(0)}M), (c,l) \mapsto A \circ (c,l)$ is smooth.
 Since $C(I, W  \times B_r^{G_{c(0)}}(0)) \cong C(I, W) \times C(I, B_r^{G_{c(0)}}(0))$ and the inclusion $\iota \colon \AC^p (I,B_r^{G_{c(0)}}(0)) \rightarrow C(I,B_r^{G_{c(0)}}(0))$ is smooth, we see that $Q_{d,p,0}^{-1} \circ (T_{\text{fib}}\tau_d^c)_* \circ (Q_{c,p,1} \times Q_{c,p,0}) = \tilde{A} (\text{inc}_I^W, \cdot) \circ \iota \times \id_{L^p (I,T_{c(0)}M)})$ is smooth. 
  \end{proof}
  Using the results obtained so far, we can copy \cite[2.3.16 Corollary]{MR1330918} to obtain the following.
 
 \begin{proposition}\label{prop:H0bun}
  Let $M$ be a strong Riemannian manifold. Then 
  \begin{displaymath}
   \partial \colon \AC^p (I,M) \rightarrow L^p (I,M\leftarrow TM), \quad \gamma \mapsto \dot{\gamma}
  \end{displaymath}
 is a smooth section of the bundle $(\pi_{TM})_* \colon L^p (I,M\leftarrow TM) \rightarrow \AC^p (I,M)$.
 \end{proposition}

 \section{The SRVT on manifolds of absolutely continuous curves}\label{sect: SRVT}
 
 We will now investigate the square root velocity transform in the framework of absolutely continuous functions.
 In the case of smooth functions, it is essential that the square root velocity transform is a diffeomorphism since one wants to construct the Riemannian metric as a pullback metric.
 As the scaling (see below) is not differentiable this is no longer possible for manifolds of absolutely continuous curves.
 However, in the vector valued case the SRVT is still a homeomorphism, whence the SRVT still relates the geodesic distances in the absolutely continuous setting.

 Our aim here is to study extensions of these results to manifold valued absolutely continuous curves. 
 In a first step we consider just an extension of the various generalised SRVT constructions for smooth maps to absolutely continuous curves. 
 It will turn out that as in the vector valued case (see \cite{bruveris2015}) only the scaling is not differentiable.
 
 \begin{setup}[Building blocks of an SRVT]
  Let us recall the three generic building blocks of any  SRVT considered so far in the literature:
 	\begin{enumerate}
 		\item \emph{Derivation} $\partial$, mapping absolutely continuous curves to their derivative.\\
 		Proposition \ref{prop:H0bun} shows that this is a smooth map from $\AC^1 (I,M)$ to a bundle of $L^1$-functions over $\AC^1 (I,M)$.\footnote{Depending on the generalisation of the SRVT, we will have to restrict to certain submanifolds of $\AC (I,M)$.}
 		\item \emph{Transport} $\alpha$, a smooth map from the $L^1$-bundle into a vector space of $L^1$-functions.
 		In the examples, the transport will turn out to be pushforward by a bundle map. 
 		\item \emph{Scaling}, Consider for some Hilbert space $E$, the (H\"{o}lder) continuous map $$\text{sc} \colon L^1 (I,E) \rightarrow L^2 (I,E),\quad  \text{sc}(v)(t) := \begin{cases}
 		\frac{v(t)}{\sqrt{\norm{v(t)}}} & \text{ if } v(t)\neq 0\\
 		0 & \text{ else}
 		\end{cases}.$$
 		is not differentiable, but a homeomorphism (see \cite{bruveris2015})
 	\end{enumerate}
 	Having chosen these building blocks such that their composition makes sense, one constructs the SRVT via 
 	$$\SRVT \coloneq \text{sc} \circ \alpha \circ \partial.$$
 \end{setup}
 
 Note that due to the derivation, the square root velocity transform will in general not be injective as $\SRVT (c)$ ``forgets'' the starting point of $c$.
 It is essential for the numerical methods to remedy this problem and there are two ways do this:
 \begin{enumerate}
  \item Choose $\star \in M$ and restrict to the closed submanifold $\AC_\star (I,M)$ of all curves which start at $\star$.\footnote{By Lemma \ref{lem: auxmap} the map $\ev_0 \colon \AC (I,M) \rightarrow M$ is a submersion, whence $\AC_\star (I,M) = \ev_0^{-1} (\star)$ is a closed submanifold. In the case where $M=G$ is an infinite-dimensional Lie group, the Lemma is not applicable. However, $\AC_\star (I,G)$ is a closed submanifold by virtue of \cite[Lemma 4.9]{hgmeasure16}.}
  \item Use the smooth map $\ev_0 \colon \AC (I,M) \rightarrow M$ to conserve the starting point, the SRVT is then constructed as a map $(\ev_0,\SRVT) \colon \AC (I,M) \rightarrow M \times B$, where $B$ is a suitable bundle of $L^2$-functions and the Riemannian metric is the product of the metrics on $M$ and $B$.
  \end{enumerate}
 
 We will now study the square root velocity transform and its generalisations from the literature. 
 In all cases it will turn out that the square root velocity transform on absolutely continuous functions can be realised as a composition of a smooth map with the (non-smooth) scaling map. In a second step we investigate then whether the $\SRVT$ actually is a homeomorphism.\footnote{Unfortunately, the author was not yet able to establish this in the Riemannian manifold case, whereas the Lie group case is unproblematic.}
 \medskip
 
 \begin{tabular}{|l | l | l | l | }
\hline {\textbf{Target manifold}} & {\textbf{transport map}} & {\textbf{Pullback of }}& {\textbf{Source}} \\ \hline 
Euclidean vector space & identity & $L^2$-metric & \cite{srivastava11sao} \\
Riemannian manifold & parallel transport 
& $L^2$-metric &\cite{su14sao} \\
Lie group  & Maurer-Cartan form & $L^2$-metric &\cite{celledoni15sao} \\ \hline
 \end{tabular}
 
 The vector space case has already been treated in literature for absolutely continuous functions (cf.\ \cite{srivastava11sao} and the summary in the Introduction). 
 A detailed analysis concerning the smoothness and homeomorphism properties of the square root velocity transform in this case can be found in \cite{bruveris2015}.
 Hence we move on to the first manifold valued case. 
 
 \subsection*{Riemannian manifolds: Parallel transport}\addcontentsline{toc}{subsection}{Riemannian manifold case: Parallel transport}

 The SRVT discussed in this section uses parallel transport to identify the derivatives of curves with vectors in a reference tangent space. Here one chooses in advance a global reference point. As mentioned in \cite{su14sao}, the behaviour of the method in a given problem depends on the choice of the reference point.
 
 For simplicity we restrict to a finite-dimensional setting, since then all results we need are readily available in the literature (and the SRVT has so far only been studied for finite-dimensional Riemannian manifolds).
 
 \begin{setup}
  Let $(M,G)$ be a finite-dimensional Riemannian manifold and $\star \in M$ be some point.
  From now on we will treat $\star$ as out reference point in the manifold and relegate qustions concerning comparisons to the tangent space over $\star$.
 \end{setup}
  
  We wish to transport vectors to the tangent space over $\star$ using parallel transport along minimal geodesics. 
  To this end, one needs to assure that the curves one is interested in pass only through points which are connected to the reference point $\star$ by a unique minimal geodesic.
  This set of points is determined by the geometry of the target manifold.
  In terms of Riemannian geometry, the open set we seek is the complement $\Omega_\star = M \setminus C(\star)$ of the cut-locus of $\star$ (cf.\ \cite[2.1.14 Theorem]{MR1330918}).
  Recall that the cut-locus is in general quite complicated (for more information see \cite[Section 2.1]{MR1330918}). 
  
  We now restrict ourselves to the open neighborhood $\Omega_\star$ of $\star$ in which there are unique minimal geodesics.
  For $p \in U_\star$ we let $c_{\star,p}$ be the unique minimizing geodesic connecting $\star$ with $p$.
  Here by minimizing geodesic we mean a geodesic with $d(c_{\star,p} (t),c_{\star,p} (s)) = |t-s|$.
  
  \begin{defn}
   Define the transport map 
    \begin{displaymath}
     \pt_\star \colon T\Omega_\star  \rightarrow T_\star M ,\quad v_p \mapsto (P^{c_{\star,p}}_{0,d(\star,p)})^{-1} (v_p)
    \end{displaymath}
   where $P^{c_{\star,p}}_{0,d(\star,p)}$ denotes parallel translation along the minimal geodesic $c_{\star, p}$ from $T_\star M$ to $T_p M$.
  \end{defn}
 
  \begin{lem}
   Let $\star \in M$ be arbitrary, then $\pt_\star$ is a smooth $T_\star M$ valued $1$-form.
   Further, the map $b_\star \colon T\Omega_\star \rightarrow \Omega_\star \times T_\star M, v_p \mapsto (p, \pt_\star (v_p))$ is a bundle isomorphism. 
  \end{lem}
 
 \begin{proof}
   Recall from \cite[1.6.2 Lemma]{MR1330918} that for fixed $c$ and $t_0,t_1$ parallel transport $P^c_{t_0,t_1}$ along $c$ is a linear isomorphism from $T_{c(t_0)} M$ to $T_{c(t_1)} M$. 
   Hence $\pt_\star$ is continuous linear on each fibre and we only have to establish continuity and smoothness of $\pt_\star$.
   
   Let $\exp \colon TM \supseteq U \rightarrow M$ be the Riemannian exponential map. 
   By definition (cf.\ \cite[2.1.4 Definition]{MR1330918}), the cut locus $C(\star)$ is the image of the infinitesimal cut locus $\widetilde{C} (\star) \subseteq T_\star M$ under $\exp_\star := \exp|_{T_\star M}$.
   Further, \cite[Theorem 2.1.14]{MR1330918} asserts that $\widetilde{C}(\star)$ is the boundary of a star shaped open neighborhood $\widetilde{\Omega}_\star$ in $T_\star M$ such that $\exp_\star (\widetilde{\Omega}_\star) = \Omega_\star$. 
   Following \cite[Section 9.1]{MR2243772} $\exp_\star$ is even a diffeomorphism which maps straight lines through the origin to minimal geodesics. 
   
   Consider the mapping 
   \begin{displaymath}
    F_\star \colon \Omega_\star \times [0,1] \rightarrow M 
    ,\quad F_\star (p,t) = \exp_\star (t\exp_\star^{-1} (p)).
   \end{displaymath}
   Clearly $R$ is smooth and since straight lines under the Riemannian exponential map get mapped to minimal geodesics, $F_\star(p,\cdot)$ is a reparametrisation of $c_{\star,p}$ by a smooth function which fixes $0$.
   Now recall from \cite[Theorem 24.1]{MR2428390} that parallel transport is reparametrisation invariant, i.e.\ $P^{F_\star(p,\cdot)}_{0,1} = P^{c_{\star,p}}_{0,d(\star,p)}$ (using that the reparametrisation fixes $0$ and $P^{c_{\star,p}}_{0,0}=\id_{T_\star M}$).
   
   For fixed $c \in C_\star^\infty ([0,1],M)$ parallel transport 
   $$P^c_{0,\cdot} \colon [0,1] \times T_{c(0)} M \rightarrow TM, (t,v_0) \mapsto P_{0,t}^c(v_0)$$
   is determined as the flow of a first order differential equation.
   Following \cite[Proof of Lemma 1.6.2]{MR1330918} the differential equation in local charts (i.e.\ $u\circ c(t) =:u(t)$) reads $$\begin{cases}
                                                                                                                                   \dot{v} &= - \Gamma(u)(\dot{u},v)\\
                                                                                                                                   v(0) &= u(v_0)
                                                                                                                                  \end{cases}$$ 
                                                                                                                                  where $\Gamma$ denotes the Christoffel symbols.
   Now replacing $u \circ c$ with $u \circ R$, it is easy to see that the right hand side of the differential equation depends smoothly on the parameter $p$. 
   It is well known (see e.g.\ \cite[Section II.9]{MR1071170}) that the solution of the differential equation $P^{F_\star(p,\cdot)}_{0,\cdot}$ depends also smoothly on the parameter $p \in \Omega_\star$.
   Summing up, we can write $\pt_\star (p) = (P^{F_\star(p,\cdot )}_{0,1})^{-1}$. 
   Parallel translation with respect to the Riemannian connection $P^{F_\star(p,\cdot)}_{0,1}$ is an isometry.
   As inversion on the subset of Banach space isomorphisms in $\mathcal{L} (T_{\star}M , T_{q}M)$ is smooth (cf.\ \cite[I \S 3, Proposition 3.9]{MR1666820}), $\pt_\star$ is smooth as a composition of smooth functions.  
   
   Now $b_\star$ is a smooth vector bundle morphism since $\pt_\star$ is a smooth $1$-form. Its inverse is given by 
   $$b_\star^{-1} \colon \Omega_\star \times T_\star M \rightarrow T\Omega_\star , \quad (p,v) \mapsto P^{c_{\star,p}}_{0, d(\star,p)} (v)$$
   Which is a smooth bundle morphism by \cite[III. \S 1 Proposition 1.3]{MR1666820} and arguments similar to the ones used to prove that $\pt_\star$ is a smooth $1$-form.  
   \end{proof}
   
   \begin{lem}\label{lem: globt:smooth}
    The maps \begin{align*}
              (\pt_\star)_* \colon L^1 (I, \Omega_\star \leftarrow T\Omega_\star) &\rightarrow L^1 (I,T_\star M),\quad h \mapsto \pt_\star \circ h \\
              (b_\star)_* \colon  L^1 (I, \Omega_\star \leftarrow T\Omega_\star) &\rightarrow \AC^p(I,\Omega_\star) \times L^1 (I,T_\star M),\quad  h \mapsto b_\star \circ h \\
              (b_\star^{-1})_* \colon \AC^p(I,\Omega_\star) \times L^1 (I,T_\star M) &\rightarrow L^1 (I, \Omega_\star \leftarrow T\Omega_\star),\quad (h_1,h_2) \mapsto b_\star^{-1} (h_1,h_2)
             \end{align*}
   are smooth.
   \end{lem}
   
   \begin{proof}
    Since $L^1 (I, \Omega_\star \leftarrow T\Omega_\star)$ carries the final topology with respect to $\{\Psi_{c,0}^{-1}\}_{c \in C^\infty (I,M)}$ it suffices to prove that $(\pt_\star)_* \circ \Psi_{c,0}^{-1}$ is smooth for every $c \in C^\infty (I,M)$.
    To simplify the notation, assume that $\exp \colon T\Omega_\star\supseteq \Omega \rightarrow \Omega_\star$ is the restriction of the Riemannian exponential map of $(M,G)$ (and $\Omega$ is chosen such that $\exp$ induces a diffeomorphism on $T_m \Omega_\star \cap \Omega$ for every $m\in \Omega_\star$. 
    Fix $\eta \in \AC_{c}^p (I,\Omega)$ and  $r>0$ with $B_r^\infty (\eta) \subseteq \Omega$. Define $\alpha (\gamma) \coloneq \alpha + \eta$ we can consider the map  
    \begin{displaymath}
     h_r \coloneq (\pt_\star)_* \circ \Psi_{c,0}^{-1} \circ (\alpha \circ Q_{c,p,1}^{-1}) \times Q_{c,p,0}^{-1} \colon \AC^p (I,B_r^{G_{c(0)}} (0)) \times L^p (I,T_{c(0)}M) \rightarrow L^p (I,T_\star M).
    \end{displaymath}
     Since addition in $\AC^p_c (I, TM)$ is smooth, it suffices to establish smoothness of $h_r$. 
     To this end consider the auxiliary map $h \colon I \times B_r^{G_{c(0)}} (0) \times T_{c(0)}M \rightarrow T_\star M , (t,x,y) \mapsto \pt_\star \circ \tau \circ (P^c_{0,t} (x) + \eta (t), P^c_{0,t} (y))$ (with $\tau$ as in \ref{setup: VBcharts}).
     Note that $h$ is smooth and $h(t,x,\cdot)$ is linear for each pair $(t,x)$. Hence we extend $h$ with Lemma \ref{lem: ted:ext} and use \cite[Proposition 2.4]{hgmeasure16} to see that $\tilde{h} \colon \AC^p (I,B_r^{G_{c(0)}} (0)) \times L^p (I,T_{c(0)}M) \rightarrow L^p (I,T_\star M), \ (\eta, \xi) \mapsto h \circ (\id_I, \eta, \xi)$ is smooth. By construction $\tilde{h} = h_r$.
     We deduce that $(\pt_\star)_*$ is smooth.
     
     Now $(b_\star)_* = ((\pi_{T\Omega_\star})_*,(\pt_\star)_*)$ is smooth by Proposition \ref{prop: postcom:sm}.
     
     Finally, we have to establish smoothness of the inverse. Again by virtue of the final topology on $\AC^p (I,\Omega_\star)$ it is enough to prove that 
     $$A_c \coloneq \Psi_{c,0} \circ (b_{\star}^{-1})_* \circ (\Phi_c^{-1} \times \id) \colon \AC^p_c(I,\Omega)\times L^1 (I,T_\star M)) \rightarrow \AC^p_c(I,\Omega) \times L^p_c (I, \Omega_\star \leftarrow T\Omega_\star)$$ 
     is smooth for every $c \in C^\infty (I,M)$. Observe that $A_c (h_1,h_2) = (h_1, A_{c,2} (h_1,h_2)$ (cf.\ proof of Lemma \ref{lem: VB:charts}).
     Hence it suffices to establish smoothness of $A_{c,2}$. Working again with the vector bundle isomorphisms $Q_{c,p,i}$ for $i\in \{1,2\}$, the argument needed to establish smoothness of $A_{c,2}$ are virtually the same as needed to establish smoothness of $(\pt_\star)_*$.
     We omit them here.
     \end{proof}

   \begin{setup}[SRVT on Riemannian manifold]
   Following \cite{su14sao} the square root velocity tranform of $c \in \text{Imm} (I,\Omega_\star)$ with respect to some choice of $\star \in M$ is given by 
    \begin{displaymath}
     \SRVT (c) := \frac{(P^{c_{\star,c(t)}}_{0,d(\star,c(t)})^{-1}\dot{c}(t)}{\sqrt{\norm{\dot{c}(t)}}}.
    \end{displaymath}
   Since parallel transport is an isometry, we can write $\SRVT (c) = \text{sc} \circ \pt_\star (\dot{c})$, where $\text{sc}$ denotes the scaling.
   Thus the square root velocity transform for Riemannian manifolds proposed in \cite{su14sao} extends to a map 
   \begin{displaymath}
    (\ev_0, \SRVT) \colon \AC^1 (I,\Omega_\star) \rightarrow \Omega_\star \times L^2 (I,T_\star M),\quad c \mapsto \text{sc} \circ (\pt_\star)_*\circ \partial (c).
   \end{displaymath}
   Following our general theme, the map decomposes into a smooth part $(\pt_\star)_*\circ \partial$ followed by the continuous scaling.
   \end{setup}
   
   It turns out that in the setting of absolutely continuous curves it is quite challenging to construct an inverse for the SRVT.
   
   \begin{setup}[Inverse of the SRVT]\label{setup: inverse SRVT}
   Let us now briefly discuss the inverse of the square root velocity transform. 
   Assume that $\gamma \in \AC^1 (I,\Omega_\star)$ with $h=\SRVT(\gamma)$. 
   Then $V(\gamma (t), h)(t)=\dot{\gamma} (t)$ whence $\gamma$ is a so called $\AC^p$-Caratheodory solution\footnote{Let $W \subseteq \R \times E$ be a subset and $f \colon W \rightarrow E$ a map. Then $\gamma \in \AC^p (I,E)$ with $t_0\in I$ is an \emph{$\AC^p$-Caratheodory solution} to $\dot{y} = f(t,y), y(t_0)=y_0$ if $(t,\gamma(t)) \in W,  \forall t \in I, \gamma(t_0) = y_0$ and $\dot{\gamma} = [t\mapsto f(t,\gamma(t))$.} to the initial value problem 
   \begin{equation}\label{eq: sol:cara}
    \begin{cases}
     \dot{\alpha}(t) = V (\alpha(t) , h)(t) \\
     \alpha (0) = \alpha_0
    \end{cases}
   \end{equation}
   for $\alpha_0 = \gamma (0)$. If we restrict our attention to smooth curves (as in \cite{su14sao}) the differential equation clearly admits a solution which depends continuously on $h$.
   Hence for the subset $C^\infty (I,T_\star M) \subseteq L^2 (I,T_\star M)$ the inverse is defined by mapping a smooth function to the solution of \eqref{eq: sol:cara}.
   To generalise this one has to establish that \eqref{eq: sol:cara} admits a unique solution. We sketch here only an argument and hope to carry out these computations in future work.
   
   Define
   \begin{displaymath}
   V \colon I \times \Omega_\star \times L^2 (I,T_\star M) \rightarrow T\Omega_\star), V(t,p,h) \mapsto b_\star^{-1} (p, \text{sc}^{-1}(h))(t)
   \end{displaymath}
   By construction $V (\cdot, p, h)$ is an $L^1$-curve in $T_pM$.
   We fix now $h$ and define $$V_h \colon I \times \Omega_\star \rightarrow T\Omega_\star, (p,t) \mapsto V(t,p,h) = P^{c_{\star,p}}_{0,t} \circ \text{sc}^{-1}(h)(t).$$ 
   As $P^{c_{\star,\cdot}}_{0,\cdot} (\cdot)$ is smooth and $\text{sc}^{-1}(h)$ is an $L^1$-function, $V_h (\cdot, p)$ is $L^1$-integrable for every fixed $p\in \Omega_\star$, so in particular measurable.
   Further , $V_h(t, \cdot)$ is continuous (even smooth!) for each $t\in I$, whence $V_h$ is a so called Caratheodory function (see \cite[Section 4.10]{MR2378491}).
   In particular, $V_h$ is jointly measurable. If we can prove now that $V_h$ is integrably Lipschitz continuous (see \cite[22.36]{MR1417259}), then \eqref{eq: sol:cara} has a unique Caratheodory solution by \cite[Theorem 30.9]{MR1417259}.
   We do not carry out these computations here, but hope to provide details in future work.
      
   Hence if we can show that the differential equations \eqref{eq: sol:cara} have unique solution for all $\alpha_0$ and $h$, then we obtain an inverse of the SRVT via 
   \begin{displaymath}
    \Omega_\star \times L^2 (I,T_\star M) \rightarrow \AC^p (I,M) , (\alpha,h) \mapsto \gamma.
   \end{displaymath}
   where $\gamma$ is the solution of \eqref{eq: sol:cara}. Continuity of this mapping can then be established by \cite[Theorem 30.10]{MR1417259}.
   \end{setup}
   
   In the next section we pass on to the case of Lie groups as target manifolds. In this setting, the problems with the inverse of the SRVT do not occur. 
   This is the reason why we conjecture that the sketch given in \ref{setup: inverse SRVT} will actually yield a continuous inverse to the $\SRVT$ (we refer to \cite[Lemma 6.3]{hgmeasure16} for the corresponding results in the Lie group setting).

 \subsection*{Lie groups: Transport via Maurer-Cartan form}\addcontentsline{toc}{subsection}{Lie groups: Transport via Maurer-Cartan form}
 In \cite{hgmeasure16} Gl\"{o}ckner has constructed a Lie group structure on Lie group valued absolutely continuous curves. 
 This Lie group structure exists even if the target Lie group is infinite-dimensional.
 In this section we describe how the square root velocity transform for Lie group valued smooth curves from \cite{celledoni15sao} extends to the absolutely continuous setting using Gl\"{o}ckners results.
 
 \begin{setup}[{\cite[Proposition 4.2 and Lemma 4.8 (c)]{hgmeasure16}}]\label{setup:LG}
  Let $G$ be a Banach Lie group. Then $\AC (I,G)$ is a group under pointwise multiplication.
  For the Lie algebra $\Lf (G)$ of $G$ the space  $\AC (I,\Lf (G))$ with the pointwise Lie bracket is a Lie algebra.
  Then one can prove: 
  
  \emph{There is a unique Lie group structure on $\AC (I,G)$ with Lie algebra $\AC (I,\Lf(G))$ such that $\AC (I,U) := \{\eta \in \AC(I,G)\mid \eta (I) \subseteq U\}$ is open in $\AC(I,G)$ and $\phi_* \colon \AC (I,U) \rightarrow \AC (I,V)$ is a diffeomorphism for each chart $\phi \colon U \rightarrow V$ of $G$ around the identity with $U = \{ g^{-1} \mid g \in U\}$. 
  Further, the inclusion $C^\infty (I,G) \rightarrow \AC (I,G)$ is a morphism of Lie groups.}\footnote{cf. \cite{MR1934608} for the Lie group structure on $C^\infty (I,G)$}.  
 \end{setup}
 
  \begin{remark}
  Note that the construction of the Lie group $\AC(I,G)$ as in \ref{setup:LG} is only a special case of the constructions outlined in \cite{hgmeasure16}.
  For example, one can consider more general target Lie groups or other regularities (e.g.\ derivatives in $L^p$). These results are beyond the scope of the present paper.
  \end{remark}
  
 On first glance, the manifold structure on the curves depends on the Lie group structure of the target manifold. 
 However, we will now argue that the manifold structure can be obtained (for finite-dimensional targets) as a special case of the manifold structure of $\AC (I,M)$ from Theorem \ref{thm: ACman}.
 
 \begin{proposition}\label{prop: HB:Lie}
  Let $G$ be a Hilbert Lie group. Then the manifold structure from Theorem \ref{thm: ACman} turns $\AC^p (I,G)$ with the pointwise operations into a Hilbert Lie group.
  Moreover, this Lie group structure coincides with the one from \ref{setup:LG}.
 \end{proposition}
 
 \begin{proof}
 Notice first that every Hilbert Lie group is a strong Riemannian manifold.\footnote{Use multiplication to take the Hilbert space product on the Lie algebra $\Lf (G) = T_e G$ to a (right) invariant metric.}
  Let $m \colon G \times G \rightarrow G$ and $i \colon G \rightarrow G$ be multiplication and inversion in the Lie group $G$. 
  Then the group operations of $\AC (I,G)$ are given by 
  \begin{align*}
   m_* \colon \AC(I,G\times G) \cong \AC (I,G) \times \AC(I,G) \rightarrow \AC(I,G),\quad i_* \colon \AC(I,G) \rightarrow \AC(I,G).
  \end{align*}
  Thus by Proposition \ref{prop: postcom:sm} the pointwise operations of $\AC(I,G)$ are smooth, whence $\AC(I,G)$ is a (Hilbert) Lie group.
  
  To see that the differentiable structure of $\AC(I,G)$ coincides with the one from \ref{setup:LG}, observe that it suffices to prove that the manifold structures coincide on a neighborhood of the identity element (i.e.\ $e_G \colon I \rightarrow G, \ t \mapsto e$) in $\AC(I,G)$. 
  Choose a chart $\phi \colon U \rightarrow V \subseteq \Lf (G)$ which satisfies the assumptions of \ref{setup:LG}. Without loss of generality we may assume that $\phi (e) = 0 \in \Lf(G)$.
  Now denote by $\rho_g \colon G \rightarrow G$ right-translation by $g$ in $G$.
  Following \cite[42.4]{MR1471480} we obtain a local addition by the following construction 
  \begin{displaymath}
   \A_G \colon TG \supseteq \bigcup_{g\in G} T\rho_g (V) \rightarrow G , \quad T \rho_g(V) \ni v_g \mapsto \rho_g \circ \phi^{-1} \circ T\rho_g^{-1} (v_g).
  \end{displaymath}
  Note that the identity element $e_G$ is contained in $C^\infty (I,G)$. 
  Thus we can construct a canonical chart $\Phi_{e_G} \colon U_{e_G} \rightarrow \AC_{e_G} (I,TG)$ around $e_G$ for the manifold structure $\AC(I,M)$ with respect to the local addition $\A_G$.
  Now the definitions yield $U_{e_G} = \AC (I,U)$ and $\Phi_{e_G} = \phi_*$, whence the manifold structure of $\AC(I,G)$ coincides with the one recalled in \ref{setup:LG}. 
 \end{proof}
 
 Recall that for a Lie group $G$ the left and right multiplication induce Lie algebra valued $1$-forms, the so called left (or right) Maurer-Cartan form.
 If we denote by $L_x$ and $R_x$ the derivative of the left (resp. right) multiplication by $x$ these are given by 
 \begin{align*}
 \kappa^l \colon TG \rightarrow \Lf (G) ,\quad v \mapsto L_{\pi_{TG} (v)}^{-1} (v), \quad \quad 
  \kappa^r \colon TG  \rightarrow \Lf (G) ,\quad v \mapsto R_{\pi_{TG} (v)}^{-1} (v).  
 \end{align*}
 Now the transport map for the square root velocity transform on $\AC^1 (I,G)$ is simply $\kappa^l_*$ or $\kappa^r_*$ (on the $L^1 (I,G \leftarrow TG)$ bundle) .
 To keep the notation as in \cite{celledoni15sao}, let us choose the right Maurer-Cartan form.
 One can establish the smoothness of the transport map using the usual techniques employed so far. However, if one wants to construct an inverse to the SRVT, one is again faced with the problem of solving a differential equation (i.e.\ integrating the derivative),
 Hence we use a different route to establish the SRVT and use Lie theory to turn it into a diffeomorphism.
 
 \begin{setup}[$L^1$-regularity for Lie groups {\cite{hgmeasure16}}]
  A Lie group $G$ with Lie algebra $\Lf (G)$ is called $L^1$-semiregular if for every $\gamma \colon I \rightarrow \Lf (G)$ there is a (necessarily unique) $\eta \in \AC^1 (I,G)$ such that 
  \begin{equation}\label{eq: rderiv}
   \delta^r (\eta) = \kappa^r_* (\dot{\eta}) = \gamma \quad \quad \text{ and } \gamma(0) = e.
  \end{equation}
  If in addition the map $\Evol \colon L^1 (I,\Lf (G)) \rightarrow \AC^1 (I,G)$ which maps $\eta$ to the solution of \eqref{eq: rderiv} is smooth, we call $G$ an $L^1$-regular Lie group. 
  Recently it has been shown that every Banach Lie group is $L^1$-regular (see \cite[Theorem C]{hgmeasure16}.
 \end{setup}
 
 Hence we can extend the square root velocity transform for Lie group valued smooth curves from \cite{celledoni15sao} to an SRVT on absolutely continuous curves.
 
 \begin{setup}[SRVT for Lie group valued absolutely continuous curves]
  Let $G$ be a Hilbert Lie group\footnote{Gl\"{o}ckner's theorem allows us to extend the SRVT on absolutely continuous curves even to Banach Lie groups. However, the smooth theory breaks down beyond the Hilbert case (see \cite{celledoni15sao}), whence we have no need for this generality.}.
  Then we define a square root velocity transform 
    \begin{displaymath}
     \SRVT \colon \AC^1 (I,G) \rightarrow L^2 (I,\Lf (G)) , \quad \gamma \mapsto \text{sc} (\delta^r (\gamma))
    \end{displaymath}
  Let now $\AC^1_* (I,G)$ be the closed submanifold of all curves starting at the identity element in $G$, then the SRVT induces a homeomorphism $\AC^1_* (I,G) \rightarrow L^2 (I,\Lf (G))$ whose inverse is given by $ L^2 (I,\Lf (G)) \rightarrow \AC^1_*(I,G), \ \eta \mapsto \Evol (\text{sc}^{-1} (\eta))$.
  Hence for Lie groups valued curves we recover the properties of the SRVT from the vector valued case. 
 \end{setup}

 \textbf{Acknowledgement} This research was supported by the European Unions Horizon 2020 research and innovation programme under the Marie Sk\l{}odowska-Curie grant agreement No.\ 691070. The author wishes to thank E.\ Celledoni and R.\ Dahmen for helpful discussions on the subject of this work.
 
 \begin{appendix}
  \section{Auxiliary results on absolutely continuous curves}\label{App: abscont}
  
   Let us first recall some basic facts on Lebesgue spaces and absolutely continuous curves.
   Our exposition here is inspired by \cite{hgmeasure16}.
   
 \begin{setup}{Conventions}
 In the following we will always assume that $a< b$ are real numbers and $(E,\norm{\cdot})$ is a Banach space and $p \in [1,\infty[$.
 $M$ will always be a Banach manifold.
 \end{setup}

 \begin{defn}\label{defn: L2}
  We let $\mathcal{L}^p ([a,b], E)$ be the set of all measurable functions $\gamma \colon [a,b] \rightarrow E$ (with respect to the Borel $\sigma$-algebras on $[a,b]$ and $E$) such that 
  \begin{displaymath}
   \norm{\gamma}_{p} \coloneq \sqrt[p]{\int_a^b \norm{(\gamma(t))}^p \dd t} < \infty
  \end{displaymath}
  For $\gamma \in \mathcal{L}^p ([a,b], E)$ we denote by $[\gamma]$ the equivalence class with respect to the relation $\gamma \sim \gamma'$ if and only if $\gamma (t) = \gamma' (t)$ almost everywhere.
  Define $L^p([a,b],E) = \{[\gamma] \mid \gamma \in  \mathcal{L}^p ([a,b], E) \}$. Then the quotient topology induced by the canonical map $ \mathcal{L}^p ([a,b], E) \rightarrow L^p([a,b],E)$ turns $L^p([a,b],E)$ into a Banach space with respect to $ \norm{\cdot}_{p}$ (cf.\ e.g. \cite[Lemma 1.19]{hgmeasure16}).
 \end{defn}

 \begin{remark}
 \begin{enumerate}
  \item The inclusions $C^\infty ([a,b],E) \subseteq C^0 ([a,b],E) \subseteq L^p([a,b],E)$ are continuous with respect to the compact-open ($C^\infty$) topology on $C^0 ([a,b],E)$ (and $C^\infty ([a.b],E)$, respectively).
        Since step functions of the form $\sum_{i=1}^n c_i 1_{X_i}$ with $c_i \in E$ and $X_i \subseteq [a,b]$ measurable, are dense in $L^p([a,b], E)$, the usual proof using mollifiers shows that $C^\infty ([a,b],E)$ is dense in $L^p([a,b],E)$. 
  \item If $(E,\langle \cdot , \cdot \rangle)$ is a Hilbert space, or more specially if $E$ is finite-dimensional, $L^2 ([a,b],E)$ is again a Hilbert space with respect to the $L^2$-inner product 
    \begin{displaymath}
     \langle f,g\rangle_{L^2} \coloneq \int_a^b \langle f(t),g(t)\rangle \dd t.
    \end{displaymath}
 \end{enumerate}\label{rem: bettertop}
 \end{remark}
 
 \begin{defn}\label{defn: AC}
  Choose $t_0 \in [a,b]$ and define $\AC^p ([a,b],E) \subseteq C([a,b],E)$ as the space of all continuous functions $\eta \colon [a,b]\rightarrow E$ for which there exists $[\gamma] \in L^p ([a,b],E)$ such that 
    \begin{equation}\label{eq: aederiv}
    \eta (t)  =  \eta (a) + \int_{a}^t \gamma (s) \dd s \quad \quad \forall t \in [a,b] .
    \end{equation}
  Here the integral denotes the weak integral (i.e.\ it is defined as the unique element which satisfies $\lambda \left( \int_{t_0}^t \gamma (s) \dd s\right) =  \int_{t_0}^t \lambda \circ\gamma (s) \dd s$ for all continuous linear $\lambda \colon E \rightarrow \R$.
  Recall that $\eta$ is then almost everywhere differentiable with derivative $\eta' = [\gamma]$ (cf.\ \cite[Lemma 1.28]{hgmeasure16}). 
   We call $\AC^p ([a,b],E)$ the space of \emph{absolutely continuous curves} (with derivative in $L^p$) and use \eqref{eq: aederiv} to turn $\AC^p ([a,b],E)$ into a Banach space such that  
    \begin{displaymath}
     \Phi \colon \AC^p([a,b],E) \rightarrow E \times L^p ([a,b],E) ,\quad  \eta \mapsto (\eta(a) , \eta')
    \end{displaymath}
  is an isomorphism of Banach spaces. The associated norm on $\AC^p([a,b],E)$ is given by $\norm{c}_{p,1} \coloneq \norm{c(a)} + \norm{\dot{c}}_p$. 
  Again, if $E$ is a Hilbert space so is $\AC^2([a,b],E)$.  
 \end{defn} 
   
   \begin{lem}\label{AC:Cinfty}
    There is a constant $R>0$ such that the supremum norm 
    \begin{displaymath}
     \norm{c}_\infty \coloneq\sup_{t \in [a,b]}\norm{c (t)} \text{ satisfies }  \norm{\cdot}_\infty \leq R\norm{\cdot}_{p,1}.
    \end{displaymath}
   Hence $(\AC^p([a,b],E),\norm{\cdot}_{p,1}) \rightarrow (C([a,b],E),\norm{\cdot}_\infty)$ is continuous and the topology of $\AC^p([a,b],E)$ is finer than the compact open topology.
   In particular, for every $V \subseteq E$ open, the subset $$\AC^p ([a,b] ,V) \coloneq \{\gamma \in \AC^p([a,b],E) \mid \gamma([a,b]) \subseteq V\}$$ is open in $\AC^p([a,b],E)$.
   \end{lem}
   
   \begin{proof}
    Observe that for $c \in \AC^p ([a,b],E)$ we have by H\"{o}lders inequality
    \begin{displaymath}
     \norm{c(t_1) - c(t_0)} \leq \int_{t_0}^{t_1} \norm{\dot{c}(t)} \dd t \leq R\norm{c}_p
    \end{displaymath}
    for some constant $R\geq 1$ which only depends on $[a,b]$. 
    Hence $\norm{c}_\infty \leq \norm{c(a)} + R \norm{\dot{c}}_p \leq R \norm{c}_{p,1}$.
   \end{proof}

   \begin{remark}
    \begin{enumerate}
     \item Remark \ref{rem: bettertop} (a) implies that $C^\infty ([a,b],E) \subseteq \AC^p([a,b],E)$ is a dense subset.
     \item A curve in $\AC^p ([a,b],\R)$ is absolutely continuous in the usual sense, i.e.\ for each $\varepsilon >0 $ exists $\delta > 0$ with $\sum_{j=1}^n |\eta (b_j)-\eta (a_j)| < \varepsilon$
   for all $n \in \N$ and disjoint subintervals $]a_i,b_i[$ of $[a,b]$ of total length $\sum_{j=1}^n (b_j -a_j) < \delta$ (see \cite[Remark 3.8]{hgmeasure16}).
    \end{enumerate}\label{rem:dense}
   \end{remark}

  Finally, we recall that for $f \colon V \rightarrow F$ smooth (and $F$ a Banach space), we have $f \circ \gamma \in\AC([a,b],F)$ for all $\gamma \in \AC([a,b],V)$.
  The resulting map 
    \begin{displaymath}
     f_* \colon \AC ([a,b],V) \rightarrow \AC ([a,b],F), \gamma \mapsto f\circ \gamma
    \end{displaymath}
 is even smooth (see \cite[Lemma 3.28]{hgmeasure16}).

   \begin{remark}\label{rem: ISO}
 The definition of the spaces $\AC^p ([a,b],E)$ and their topology used an integral depending on $a$.
 One can show that neither the definition nor the topology depend on this choice. Instead, one could replace $a$ in both cases by any $t_0 \in [a,b]$ without changing the space or its topology (cf.\ \cite[Remark 3.10]{hgmeasure16}).
 \end{remark}
  \subsection*{Topologies on sections over absolutely continuous curves}\addcontentsline{toc}{subsection}{Sections of vector bundles}
   In this section we recall the construction of topologies on spaces of absolutely continuous sections.
   Note that the classical arguments have to be adapted for infinite-dimensional target manifolds.
   To simplify the notation let us agree on the following Conventions:
   
   \begin{setup}
   Throughout this section, $(M,G)$ will be a strong Riemannian manifold with Levi-Civita covariant derivative $\nabla$. 
   For $c \in \AC^p (I,M)$ we use the shorthand $\nabla_c \coloneq \nabla_{\frac{\dd c}{\dd t}}$.\footnote{Note that the usual proof \cite[VIII \S 3 Theorem 3.1]{MR1666820} covariant derivatives along an absolutely continuous curve make sense for lifts over that curve and that $\nabla_c \eta$ is an $L^p$-curve in $TM$ over $c$ if $\eta$ is an $\AC^p$-curve. }  
   Further, we let $M$ be modelled on the Banach space $(E,\norm{\cdot})$ and fix $p \in [1,\infty[$. 
   \end{setup}

  Our aim here is to topologise the spaces $\AC_\gamma^p ([a,b],TM)$ and $L^p_\gamma ([a,b], M \leftarrow TM)$ from Definition \ref{defn: asec} and \ref{defn:L2sect}.
  Using the Riemannian structure of $M$ we let $\norm{\cdot}_{x}$ be the norm induced on $T_xM$ by $G$ for $x \in TM$.
  Then we can define 
  \begin{align}
   \norm{X}_{\gamma , p,0} &:= \left(\int_a^b \norm{X(t)}_{\gamma(t)}^p \dd t\right)^{\frac{1}{p}}  \quad \text{for } X \in L^p_\gamma ([a,b], M\leftarrow TM) \label{IP: L2}\\
   \norm{X}_{\gamma , p,1} &= \norm{X(a)}_{\gamma(a)} +  \norm{\nabla_\gamma X}_{\gamma , p,0}\quad \text{for } X \in \AC_\gamma^p ([a,b],TM) \label{IP: AC1}
  \end{align}
 We will show that these mappings are norms turning the sections into Banach spaces. 
  Finally,  
 $$\norm{X}_{\infty,\gamma} := \sup_{t\in [a,b]} \norm{X(t)}_{\gamma(t)} \quad X \in \AC_\gamma^p ([a,b],B).$$
 We will now construct a bundle trivialisation which induces an isometry of the above norms to certain Banach spaces. 
 Hence, we see that a posteriori the space of sections with these norms are Banach/normed spaces.
 
 Let $c \in \AC^p ([a,b],M)$ and denote for $t_0,t_1 \in [a,b]$ by $P^c_{t_0,t_1} \colon T_{c(t_0)}M \rightarrow T_{c(t_1)}M$ parallel transport along $c$ in the tangent bundle (with respect to the Riemannian structure). 
 Since it is non-standard, let us assure that parallel transport along absolutely continuous curves is indeed well-defined and an isometry.\footnote{The class of absolutely continuous curves is the most general class for which this makes sense.} As explained in \cite[II. Satz 2.2]{FlaKli72} the usual proof (see e.g.\ \cite[VIII. \S 3]{MR1666820} carries over to $\AC^p$ curves by virtue of a suitable solution theory for ODE's (see \cite[Section 30]{MR1417259} for results on Banach spaces ). 
 Using parallel translation, one can show as in \cite[II. Satz 2.3 and Bemerkung 2.4]{FlaKli72} the following (the local descriptions are available by \cite[VIII \S 2]{MR1666820} and the ODE solution theory can be found in \cite[Section 30]{MR1417259}.).  
 
 \begin{setup}\label{top:sectsp} Let $c \in \AC^p ([a,b],M)$ and $(M,G)$ a strong Riemannian manifold $(M,G)$.
  \begin{enumerate}
  \item Then the maps
  \begin{align*}
   Q_{c,p,0} \colon (L^p_{c} ([a,b],M \leftarrow TM), \norm{\cdot}_{c,p,0}) &\rightarrow L^p ([a,b], T_{c(a)}M),\\ X &\mapsto (t\mapsto P_c|_{[t,a]} \circ X(t))\\
   Q_{c,p,1} \colon (AC^p_{c} ([a,b],TM), \norm{\cdot}_{c,p,1}) &\rightarrow \AC^p ([a,b], T_{c(a)}M),\\ X &\mapsto (t\mapsto P_c|_{[t,a]} \circ X(t))
  \end{align*}
  make sense and we have $\frac{\dd}{\dd t} \circ Q_{c,p,1} = Q_{c,p,1} \circ \nabla_c$.
  Thus the maps $Q_{c,p,i}, i=1,2$ become isomorphisms of Banach spaces.
  \item Further, $Q_{c,p,1}$ induces an isomorphism of normed spaces 
  $$(\AC_c^p ([a,b],TM), \norm{\cdot}_{\infty,c}) \rightarrow (\AC^p ([a,b], T_{c(a)}M), \norm{\cdot}_{\infty}).$$
  \end{enumerate}
  Then for each $\varepsilon >0$ the set $B^\infty_\varepsilon (X) := \{ Y \in \AC_c^p ([a,b],B) \mid \norm{Y-X}_{\infty,c} < \varepsilon\}$ is open in $(AC_{c}^p ([a,b],TM), \norm{\cdot}_{c,p,1})$.
  In particular, Lemma \ref{AC:Cinfty} shows that the topology of $AC_{c}^p ([a,b],TM)$ is finer than the one of $(\AC_c^p ([a,b],TM), \norm{\cdot}_{\infty,c})$.
  \end{setup}
  
  \begin{lem}
   Let $(M,G)$ be a strong Riemannian manifold, $p \in [1,\infty[$. The topologies of 
    \begin{enumerate}
     \item $(L^p_{c} ([a,b],M \leftarrow TM),  \norm{\cdot}_{c,p,0})$, $(\AC^p_{c} ([a,b],TM), \norm{\cdot}_{\infty,c})$ and
     \item $(\AC^p_{c} ([a,b],TM), \norm{\cdot}_{c,p,1})$
    \end{enumerate}
  constructed in \ref{top:sectsp} do not depend on the choice of the Riemannian structure.
  \end{lem}

  \begin{proof}
   Let $\widetilde{G}$ be another strong Riemannian metric on $M$.
   Since the metrics are strong, for each $t\in [a,b]$ the topologies on $T_{c(t)} M$ induced by $G_{c(t)}$ and by $\widetilde{G}_{c(t)}$ coincide with the natural topology of the tangent space.
   
   \begin{enumerate}
    \item  Hence for every fixed $t$ the norms induced by $G$ and by $\widetilde{G}$ are equivalent.   
   Now we can argue as in \cite[II. Bemerkung 2.4 (iii)]{FlaKli72}: Due to compactness of $c([a,b])$ we take the maximum/minimum over the equivalence constants to see that the norms $\norm{\cdot}_{c,p,0}$ induced by $G$ and $\tilde{G}$ coincide.
   The same argument shows that the norms of type $\norm{\cdot}_{\infty,c}$ are equivalent.
   \item  Let $\widetilde{\nabla}$ be the covariant derivative with respect to $\widetilde{G}$.
  In view of (a), it suffices to prove that for each $X \in \AC_c^p (I, TM)$ we have $\norm{\widetilde{\nabla}_{c} X}_{c,p,0}  \leq C\norm{X}_{c,p,1}$ for some constant $C>0$.
     Now we obtain the estimate 
     \begin{equation} \label{est:eq1} \begin{aligned}
     \norm{\widetilde{\nabla}_{c} X }_{c,p,0} &\leq \norm{\nabla_{c}X}_{c,p,0} + \norm{(\widetilde{\nabla}_{c} -\nabla_{c})X}_{c,p,0}\\
						&\leq \norm{X}_{c,p,1} + \left(\int_{a}^b \norm{(\widetilde{\nabla}_{c} -\nabla_{c})X(t)}_{c(t)}^p\right)^{\frac{1}{p}}	
						\end{aligned}				
   \end{equation}
   Thus if we can bound the integral with some multiple of $\norm{X}_{c,p,1}$, we are done. 
   
   Let us assume at first that $c([a,b])$ is contained in a manifold chart $\kappa \colon U \rightarrow V \subseteq E$. 
   Since $c([a,b])$ is compact, we can choose and fix $k,K >0$ such that for every $t \in [a,b]$ we have
   \begin{align}
    k\norm{T_{c(t)}\kappa (v)} \leq \norm{v}_{c(t)} \leq K \norm{T_{c(t)}\kappa (v)}.\label{norm:equiv}
   \end{align}
   Hence we can replace $\norm{\cdot}_{c(t)}$ with the Hilbert space norm to obtain an estimate.
   To this end, we adopt the notation of \cite[Chapter VIII]{MR1666820}: Local representatives of geometric objects, e.g. a vector field $\xi$, in the chart $\kappa$ will be labeled as $\xi_U$
   Now by \cite[VIII \S 2 and VIII \S 3 Theorem 3.1]{MR1666820} there exists a smooth map $B_U \colon [a,b] \rightarrow \mathcal{B} (E,E),\ t \mapsto B_U (c(t); \dot{c}_U(t) ,\cdot)$ ($\mathcal{B} (E,E), \norm{\cdot}_{\text{op}}$ bounded linear operators with the operator norm) such that 
   $$\nabla_{c} \xi (t) = \dot{\xi}_U (t) - B_U (c(t),\dot{c}_U(t), \xi_U (t)).$$
   We let $\widetilde{B}_U$ be the map defined similarly for $\widetilde{\nabla}$
   Hence in local charts the integral becomes 
   \begin{align*}
    &\int_{a}^b \norm{(\widetilde{\nabla}_{c} -\nabla_{c})X(t)}_{c(t)}^p \\
    \stackrel{\eqref{norm:equiv}}{\leq}& K^p  \int_{a}^b \norm{B_U (c(t),\dot{c}_U(t), X_U (t)) - \widetilde{B}_U (c(t),\dot{c}_U(t), X_U (t))}^p \dd t \\
    \stackrel{\hphantom{\eqref{norm:equiv}}}{\leq}& K^p  \int_{a}^b \left(\underbrace{\norm{X_U(t)}}_{\leq \norm{X_U}_\infty} \underbrace{\norm{B_U (c(t); \dot{c}_U (t) , \cdot) - \widetilde{B}_U (c(t);\dot{c}_U (t) ,\cdot)}_{\text{op}}}_{< L < \infty}\right)^p \dd t\\
    \stackrel{\eqref{norm:equiv}}{\leq}& \left(\frac{KL}{k}\right)^p \norm{X}_{\infty,c}^p \int_a^b \dd t \stackrel{\text{Lemma } \ref{AC:Cinfty}}{\leq}  \left(\frac{KLR}{k}\right)^p (b-a) \norm{X}_{c,p,1}^p 
    \end{align*}
   Here we have used that $[a,b]$ is compact, to construct the constant $\infty > L>0$. The last inequality exploited that $Q_{c,p,1}$ is an isomorphism of Banach spaces. 
   Taking the $p$th root of the above expression and inserting it into \eqref{est:eq1} we obtain the desired estimate.
   Thus we derive the desired estimate if $c([a,b])$ is contained in a chart. 
   For the general case, we simply divide the interval $[a,b]$ into compact subintervalls such that each subintervall is mapped by $c$ into a chart. Repeating the estimate for every subintervall (and noting that the norms over a subintervall are bounded above by the norm over $[a,b]$), we can take the maximum of the estimates to obtain an estimate in the general case. \qedhere
   \end{enumerate}
  \end{proof}

    \begin{remark} For $p=2$ these norms are again induced by the obvious inner products:
   \begin{align}
    \langle X,Y\rangle_0 &:= \int_a^b g_{\gamma(t)} (X(t),Y(t)) \dd t \\
    \langle X,Y\rangle_1 &:= g_{c(a)} (X(a),Y(a)) + \int_a^b g_{\gamma (t)} (\nabla_\gamma X (t),\nabla_\gamma Y(t)) \dd t 
   \end{align}
  \end{remark}

  \begin{setup}[Curves into the Whitney sum]\label{setup: Whitney}
   Let $(M,G)$ be a strong Riemannian metric. Then the Whitney sum $TM \oplus TM \rightarrow M$ (i.e.\ the bundle given fibre-wise by $(TM \oplus TM)_x \coloneq T_xM \times T_xM$) inherits a strong Riemannian metric from $(M,G)$ (i.e.\ the one induced by $G_x \times G_x$ on $T_x M \times T_xM$).
   Fix $c \in C^\infty ([a,b],M)$ and endow $\AC^p ([a,b], T_{c(0)}M \times T_{c(0)}M)$ with the norm $\norm{\cdot}_{c,p,1}$ constructed with respect to the Banach space $T_{c(0)}M \times T_{c(0)}M$.\footnote{As is well known there are many ways to obtain (equivalent) norms on a product of Banach spaces which turn the product into a Banach space. For simplicity we will always choose $\norm{(x,y)}_{E\times F} \coloneq \max \{\norm{x}_E, \norm{y}_F\}$ as norm on the product.} Then $$Q_{c,p,1} \oplus Q_{c,p,1} \colon \AC^p ([a,b], T_{c(0)}M \times T_{c(0)}M) \rightarrow \AC^p_c ([a,b], TM \oplus TM)$$ induces a Banach space structure on $\AC^p_c ([a,b],TM\oplus TM)$ (which follows from arguments as in \ref{top:sectsp}, cf.\ \cite[II. 3.6]{FlaKli72} for more details.).
   Since $\AC^p (I, T_{c(0)}M \times T_{c(0)}M) \cong \AC^p (I, T_{c(0)}M) \times  \AC^p (I, T_{c(0)}M)$, we obtain a canonical identification $\AC^p_c(I, \Omega) \times \AC^p_c (I, TM) \cong \AC^p_c (I,\Omega \oplus TM)$ of Banach spaces. 
   Similarly, we can construct a Banach space $L^p_c ([a,b], M \leftarrow TM\oplus TM)$ using $\norm{\cdot}_{c,p,0}$ and the isomorphism $Q_{c,p,0} \oplus Q_{c,p,0}$.
   \end{setup}

   \begin{proposition}\label{prop: ACOmega:Rep}
  	Let $O \subseteq c^*TM$ be an open set with $O_t \coloneq O \cap (c^*TM)_t \neq \emptyset$ for all $t \in [a,b]$.
  	Further, we fix a smooth fibre-preserving map $F \colon c^*TM \supseteq O \rightarrow g^*TM.$
  	Then the map 
  		\begin{align*}
  		F_* \colon \AC_c^p ([a,b],O) \rightarrow \AC_g^p ([a,b],TM),\quad h \mapsto F \circ h   
  		\end{align*}
  		is smooth (where we suppressed the identification $\AC^p_{\id_I} (I,f^*TM) \cong \AC^p_f(I,TM)$). We let $T_{\text{fib}} F$ be the fibre-derivative of $F$, i.e.\ for $F_t \colon T_{c(t)}M \rightarrow T_{g(t)}M$ the fibre-derivative is $T_{\text{fib}} F \colon TM \oplus TM \rightarrow TM,\ (x,y) \mapsto d F_t (x;y)$ for $x,y \in T_{c(t)} M$. Idenfifying as in \ref{setup: Whitney}, we have $d(F_*) = (T_{\text{fib}}F)_*$.
  \end{proposition}
  
  \begin{proof}
 	    Since smoothness is a local property, it suffices to prove that $F_*$ is smooth on a neighborhood $W_\xi$ for each $\xi \in \AC^p_c ([a,b],O)$.
  		Since $\xi ([a,b])$ is compact and $O$ is open in $c^*TM$, there is $r>0$ such that $W_\xi \coloneq B_r^\infty (\xi) \subseteq \AC^p_c ([a,b],O)$.
  		Now $\AC^p_c ([a,b],O)$ is a Banach space, whence $\tau_\xi \colon \AC^p_c ([a,b],O) \supseteq B_r^\infty (0_c) \rightarrow W_\xi, \gamma \mapsto \gamma +\xi$ is a diffeomorphism (where $0_c$ is the zero section over $c$).
  		Hence it suffices to prove that $F_* \circ \tau_\xi$ is smooth. 
  		Consider the isometries $Q_{c,p,1}$ and $Q_{g,p,1}$ induced by parallel transport (cf.\ \ref{top:sectsp}).
  		Now $(P_{0,\cdot}^c)^{-1} ([a,b] \times B_r^\infty (0_c) )= B_r^{G_{c(0)}} (0) \subseteq T_{c(t)} M$ and we conclude that 
  		$$Q_{g,p,1}^{-1} \circ F_* \circ \tau_\xi \circ Q_{c,p,1} \colon \AC^p_c ([a,b], B_r^{G_{c(0)}} (0)) \rightarrow \AC^p_g ([a,b],T_{g(0)} M).$$
  		To establish smoothness consider the auxiliary mapping 
  		\begin{equation}\label{eq: aux:smoothness}
  		h \colon [a,b] \times B_r^{G_{c(0)}} (0) \rightarrow T_{g(0)}M , \quad (t,x) \mapsto (P^g_{0,t})^{-1} \circ F \circ P^c_{0,t} (x + (P^c_{0,t})^{-1} (\xi(t))).
  		\end{equation}
  		By construction $h$ is smooth and we have $h(t,\eta) = Q_{g,p,1}^{-1} \circ F_* \circ \tau_\xi \circ Q_{c,p,1} (\eta) (t)$.
  		Choose an open neighborhood $W$ of $[a,b]$ together with a smooth extension $H \colon W \times B_{r}^{G_{c(0)}} (0) \rightarrow  T_{g(0)}M$ of $h$ (this is possible by Lemma \ref{lem: ted:ext} below). Now by \cite[Lemma 3.30]{hgmeasure16} the map 
  		\begin{displaymath}
  		\tilde{H} \colon \AC^p ([a,b],B_{r}^{G_{c(0)}} (0)) \rightarrow \AC^p ([a,b],T_{g(0)}M) , \eta \mapsto H \circ (\id_{[a,b]},\eta) 
  		\end{displaymath}
  		is smooth. As $\tilde{H} = Q_{g,p,1}^{-1} \circ F_* \circ \tau_\xi \circ Q_{c,p,1}$, $F_*$ is smooth.
  		
  		To establish the formula for the derivative observe now that by the chain rule 
  		$$d F_* = d(F_* \circ \tau_\xi) = d\left( Q_{g,p,1} \circ \tilde{h} \circ Q_{c,p,1}^{-1}\right) = d Q_{g,p,1} \circ T\tilde{h} \circ TQ_{c,p,1}^{-1}.$$
  		Here we have used that the derivative of the translation $\tau_\xi$ is the identity.
  		As $Q_{c,p,1}$ is a Banach space isomorphism, we have $T Q_{c,p_1} = (Q_{c,p,1} \circ \text{pr}_1, dQ_{c,p,1}) =  Q_{c,p,1} \times Q_{c,p,1}$ (using that $\AC^p ([a,b], E\times F) \cong \AC^p ([a,b], E) \times \AC([a,b],F)$).
  		Now arguing as in \cite[Lemma 3.28]{hgmeasure16} the tangent map $T\tilde{h} = (\tilde{h},d\tilde{h})$ satisfies $d\tilde{h} = \widetilde{d_2 h}$, where 
  		\begin{displaymath}
  		 d_2 h \colon [a,b] \times \left(T_{c(0)} M \times T_{c(0)} M \right) \rightarrow T_{g(0)} M,\ (t,x,y) \mapsto \lim_{s \rightarrow 0} s^{-1} (h(t,x+sy)-h(t,x)),
  		\end{displaymath}
  		and $ \widetilde{d_2 h}$ is constructed analogous to $\tilde{h}$.
  		Let us now compute $d_2 h$. Denote by $F_t \colon B_r^{G_{c(t)}} (0) \rightarrow T_{g(t)}M$ the restriction of the fibre-preserving map $F \circ \tau_\xi$. 
  		As parallel transport is linear in the fibre, we have $d_2 h (t,x;y) = P_{t,0}^{g} (d F_t(P_{0,t}^{c}(x);P_{0,t}^{c} (y)))$. 
  		Thus $\widetilde{d_2h} = Q_{g,p,1}^{-1} \circ (T_\text{fib} F)_*\circ Q_{c,p,1}\times Q_{c,p,1}$.
  		Plugging this into the formula for $d F_*$ we obtain 
  		\begin{displaymath}
  		dF_* = Q_{g,p,1} \circ \left( Q_{g,p,1}^{-1} \circ (T_\text{fib} F)_* \circ Q_{c,p,1} \times Q_{c,p,1}\right) \circ Q_{c,p,1}^{-1} \times Q_{c,p,1}^{-1} = (T_\text{fib} F)_*
  		\end{displaymath}
  \end{proof}

  \begin{remark}\label{rem: ext:ACProP}
   The statement of Proposition \ref{prop: ACOmega:Rep} extends verbatim to smooth fibre-preserving mappings $f^*B \subseteq O \rightarrow g^D$, where $B,D$ is one of the bundles $TM$ or $TM \oplus TM$.
   Using the preparations in \ref{setup: Whitney}, the proof of Proposition \ref{prop: ACOmega:Rep} carries over almost verbatim if one replaces $Q_{c,p,i}$ with $Q_{c,p,i} \oplus Q_{c,p,i}$ ($i \in \{0,1\})$ whenever one of the bundles is of type $TM\oplus TM$.
   Even stronger one can generalise Proposition \ref{prop: ACOmega:Rep} to bundles of the form $f^*B$, where $B \rightarrow M$ is an arbitrary vector bundle with a strong Riemannian metric. Since we have no need for these results, we chose to avoid technicalities and consider only $TM, TM \oplus TM$.
   See however \cite[Lemma 2.3.9.]{MR1330918} for a proof of the full statement in the case $p=2$ and $\dim M < \infty$.
  \end{remark}

    \begin{lem}\label{lem: ted:ext}
   	Let $\gamma \colon [a,b] \times U \rightarrow F$ be a smooth map, where $a<b$, $E,F$ are Banach spaces and $U \subseteq E$ is open. Then there exists an open neighborhood $[a,b] \subseteq W$ together with a smooth extension $\overline{\gamma} \colon V \times U \rightarrow F$ of $F$. 
   	If in addition $E = E_1  \times E_2$ and $U = Z \times E_2$ for $Z  \subseteq E_1$ open and $\gamma(t,z,\cdot)$ is linear for each $(t,z) \in [a,b] \times Z$, we can choose the extension such that $\overline{\gamma}(s,z,\cdot)$ is linear for each $(s,z) \in W \times Z$.
    \end{lem}

 \begin{proof}
 	Since $[a,b]$ is a closed convex subset of $\R$ we can choose a smooth collar, i.e.\ a diffeomorphism $\theta \colon \{a,b\} \times \R \rightarrow V$ onto an open neighborhood $V \subseteq \R$ of $\{a,b\}$ which is the identity on $\{a,b\} \times \{0\}$ and $\theta (\{a,b\} \times \{t \in \R \mid t \geq 0 \}	)= [a,b] \cap V$ (cf. \cite[Proposition 24.9]{MR1471480}). We set $W \coloneq [a,b] \cup V$ (which is an open neighborhood of the compact interval).
 	Now the smooth map $$\Theta \colon (\{a,b\} \times U) \times \R \rightarrow V \times U, \quad (x,y,t) \mapsto \theta ((x,t),y)$$
 	clearly defines a collar for $[a,b] \times U$ on the open neighborhood $V \times U$ of the boundary.
 	For $x,y \in V \times U$ we observe that $\Theta^{-1} (x,y) = ((p_x,y),t_x)$, where $(p_x,t_x) = \theta^{-1} (x)$.
 	Now recall from \cite[Lemma 16.8]{MR1471480} that the restriction map $C^\infty (\R ,\R) \rightarrow C^\infty ([0,\infty[ , \R)$ admits a continuous linear section $S$.
 	This allows us to define 
 	$
 	 \overline{\gamma} (x,y) \coloneq \begin{cases}
 	 \gamma (x,y) & \text{if } x \in ]a,b[\\
 	 S(\gamma (\Theta (p_x,y,t_x))) & \text{if } x \in V
 	 \end{cases}
 	$
 	.
 	By \cite[Theorem 24.8]{MR1471480}, $\overline{\gamma}$ map makes sense, is smooth and extends $\gamma$ to $W \times U$. 
    The formula for $\overline{\gamma}$ implies that $\overline{\gamma} (s,z,\cdot)$ is linear for all $(s,z) \in W \times Z$ if $\gamma (t,z,\cdot)$ is linear for all $t \in [a,b]$.  
 \end{proof}
    
 \section{Auxiliary results in local charts}
 
 In this appendix we prove several auxiliary results used in Section \ref{sect: MFDabs}. 
 For finite dimensional Riemannian manifolds these results are easy consequences of standard facts which can be found in most entry level textbooks on Riemannian geometry. 
 
 \begin{setup}[Conventions]
  Throughout this appendix, $(M,G)$ will be a strong Riemannian manifold modelled on the Hilbert space $(H,\langle \cdot , \cdot \rangle)$.
  Denote by $\norm{\cdot}$ the norm of $H$ and for $x \in M$ let $\norm{\cdot}_{G,x}$ be the norm induced by $G$ on $T_xM$.
  Let $\kappa \colon U \rightarrow V \subseteq H$ be a manifold chart. 
  Further we fix an absolutely continuous curve $\gamma \colon [a,b] \rightarrow U$ (for $a<b$ real numbers).
  \end{setup}

 \begin{lem}\label{lem: normcomp}
  There exists an open neighborhood $W_\gamma \subseteq U$ of $\gamma ([a,b])$ together with constants $k,K >0$ such that 
  \begin{equation}\label{eq: est:normcomp}
   \forall x \in W_\gamma , v \in T_xM \quad k\norm{T_x \kappa (v)} \leq \norm{v}_{G,x} \leq K \norm{T_x \kappa (v)}
  \end{equation}
 \end{lem}
 
 \begin{proof}
  Recall from \cite[VII \S 2 Proposition 2.5 and \S 3 Theorem 3.1]{MR1666820} that there is a smooth map $B \colon U \rightarrow \mathcal{B}^\times (H,H),\ x \mapsto B_x$ (where $\mathcal{B}^\times (H,H)$ is the Banach space of invertible bounded linear mappings) such that 
  \begin{align*}
   \norm{v}_{G,x} = \norm{B(x)T_x\kappa (v)} \leq \norm{B(x)}_{\text{op}} \norm{T_x \kappa (v)} \quad \forall x \in U ,\ v \in T_x M, \text{ whence }\\ 
   \norm{T_x \kappa (v)} = \norm{B(x)^{-1}B(x)T_x \kappa (v)} \leq \norm{B(x)^{-1}}_{\text{op}}  \norm{v}_{G,x}
  \end{align*}
  As inversion in $\mathcal{B}^\times (H,H)$ is smooth (see \cite[I \S 3 Proposition 3.9]{MR1666820}), the maps $f \colon U \rightarrow \R,\ f(x) \coloneq \norm{B(x)}_{\text{op}}$ and $F\colon U \rightarrow \R,\ F(x) \coloneq \norm{B(x)^{-1}}_{\text{op}}$ are continuous.
  Now $\gamma([a,b])$ is compact, whence $f,F$ are bounded on this set, say $f \leq r$ and $F \leq R$ (with $r,R>0$ constant).
  By continuity $W_\gamma \coloneq f^{-1} (]0,2r[) \cap F^{-1} (]0,2R[)$ is an open neighborhood of $\gamma ([a,b])$ in $U$. 
  By construction, $k\coloneq \frac{1}{2r}$ and $K \coloneq 2R$ satisfy the estimates \eqref{eq: est:normcomp} on $W_\gamma$. 
 \end{proof}

 \begin{lem}\label{lem: pieces}
  Fix a local addition $\A \colon \Omega \rightarrow M$. Then there exists a smooth curve $f \colon [a,b] \rightarrow U$ with $f(a) = \gamma (a)$ and $f(b) = \gamma (b)$ such that $(f,\gamma)(t) \in (\pi_{TM} ,\A) (\Omega)$ for all $t \in I$.
  For $\varepsilon>0$ we can choose $f$ such that $\sup_{t \in [a,b]}\norm{(\pi_{TM} ,\A)^{-1} (f,\gamma)(t)}_{G , f(t)} < \varepsilon$. 
 \end{lem}

 \begin{proof}
  Set $\Omega_U \coloneq TU \cap \A^{-1} (U)$ and consider $\A_{U} \colon V \times H \supseteq T\kappa (\Omega_U) \rightarrow H , \ (x,y) \mapsto \kappa \circ \A \circ T_x\kappa^{-1}(y)$, i.e.\ the representative of the local addition in the chart $\kappa$.
  
  By construction $(\pr_1,\A_U) \colon T\kappa (\Omega_U) \rightarrow V\times V$ induces a diffeomorphism onto a neighborhood $W$ of the diagonal. 
  Denote by $h \colon V \times V \supseteq W \rightarrow h(W) \subseteq T\kappa (\Omega)$ its inverse.  
  Consider the compact set $K \coloneq \kappa \circ\gamma([a,b]) \subseteq V$. 
  Then $\Delta (K) \coloneq \{(\gamma(t),\gamma(t)) \mid t\in [a,b]\}$ is a compact subset of $W$ and $h(\Delta (K) = K \times \{0\}$ is also compact. 
  Applying \cite[Chapter II, Proposition 6.3 and 6.4]{MR1071170} to $h$ we can shrink $W$ to a neighborhood of the compact set $\Delta (K)$ such that $h$ is uniformly Lipschitz continuous on $W$ and $W \subseteq W_\gamma$ for $W_\gamma$ as in Lemma \ref{lem: normcomp}.  
  
  An application of Wallace Theorem \cite[Theorem 3.2.10.]{MR1039321} shows that there are $\varepsilon_1, \varepsilon_2 >0$ such that for $K + B_{\varepsilon_1}(0) \coloneq \{v\in V \mid \exists k \in K \norm{v-k}<\varepsilon_1\}$
  we have $(K+V_{\varepsilon_1} (0)) \times B_{\varepsilon_2} (0) \subseteq h(W)$. 
  Since $h$ is uniformly Lipschitz continuous, we see that for every $\varepsilon_2 > r >0$ there is $\delta>0$ with the following property for all $x,y \in K + V_{\varepsilon_1} (0)$ with $\norm{x-y}<\delta$ we have $h(x,y) \in \{x\} \times B_{r} (0)$.
  Conversely, this implies $\{x\} \times B_\delta (x) \subseteq h^{-1} (\{x\} \times B_{r} (0))$.
  
  Now $B_\delta^\infty (\kappa \circ \gamma) \coloneq \{g \in \AC^p ([a,b],V) \mid \sup_{a \leq t \leq b}\norm{g(t)-\gamma(t)}\}$ is an open set in $\AC^p([a,b],V)$.
  Since smooth mappings are dense in $\AC^p([a,b],V)$ we find a smooth map $g$ in $B_\delta^\infty (\kappa \circ \gamma)$. 
  Using convexity of $B_\delta (\kappa \circ \gamma(a)$ and $B_\delta (\kappa \circ \gamma(b))$ and cut-off functions, we may clearly choose $g$ such that $g(a)=\kappa \circ \gamma(a)$ and $g(b)=\kappa \circ \gamma(b)$.  
  By construction we have for each $t \in [a,b]$ the relation 
  \begin{equation}\label{loc: locadd:eq}
  (g(t),\kappa \circ \gamma(t)) \in f(t) \times B_{\delta} (g(t)) \subseteq h^{-1}(g(t),B_r (0))
  \end{equation}
  Recall now that $h^{-1}$ is the local representative of the local addition. Hence \eqref{loc: locadd:eq} implies that $f \coloneq \kappa^{-1} \circ g$ satisfies $(f,\gamma)(t) \in (\pi_{TM} ,\A) (\Omega)$.
  
  To see that also the additional condition can be satisfied fix $\varepsilon>0$.
  Since $W \subseteq W_\gamma$, we use the constants from Lemma \ref{lem: normcomp} to choose $r >0$ so small that $T\kappa^{-1} (W \times B_r (0)) \subseteq \bigcup_{x \in \kappa^{-1} (W)} B_\varepsilon^{G_x} (0)$. 
  Then the construction of $f$ proves that $\norm{(\pi_{TM} ,\A)^{-1} (f,\gamma)(t)}_{G_{f(t)}} < \varepsilon$.
 \end{proof}

  \begin{lem}\label{lem: sandwich}
  	Fix a local addition $\A \colon TM \supseteq \Omega \rightarrow M$ and a manifold chart $\varphi \colon Y \rightarrow Z$ of $M$. Let $C>0$ be constant and $f \colon [a,b] \rightarrow Y \subseteq M$ be absolutely continuous. Assume that $(f,\gamma) \subseteq (\pi_{TM},\A)(\Omega)$. Then there are constants $r,s>0$ such that $B_r^\infty ((\pi_{TM},\A)^{-1}(f,\gamma)) \subseteq \AC^p_f ([a,b],\Omega)$ and for all $t\in [a,b]$ we have 
  	\begin{align*}
  	\kappa^{-1} (B_s (\kappa \circ \gamma (t))) \subseteq  \A (B_r ((\pi_{TM},\A)^{-1}(f,\gamma)(t))
  	\subseteq \kappa^{-1} (B_C (\kappa \circ \gamma (t))).
  	\end{align*}
  \end{lem}
  
  \begin{proof}
  	Shrinking the chart $Y$, we may assume $Y \subseteq W_f$ where $W_f$ is the set constructed in Lemma \ref{lem: normcomp}. Thus Lemma \ref{lem: normcomp} allows us to argue in local charts and replace the norm of the Riemannian metric by the norm of $H$. 
  	
  	In the following we suppress the charts and assume that $f \colon [a,b] \rightarrow Z \subseteq H$ and $\gamma \colon [a,b] \rightarrow V \subseteq H$ are absolutely continuous curves.
  	Then the local addition corresponds to a map 
  	$\psi \colon Z \times H \supseteq \Omega_\psi \rightarrow V$ such that $(\pr_1,\psi)$ induces a diffeomorphism onto a neighborhood of $\{(f(t),\gamma(t)), t\in I\}$. 
  	Denote by $f_\gamma \colon [a,b] \rightarrow E$ the unique curve with $\psi \circ (f, f_\gamma) = \gamma$
  	
  	Now $(f,f_\gamma)([a,b])$ and $(f,\gamma)$ are compact sets which are mapped by the diffeomorphism $(\pr_1,\psi)$ to each other. Hence we can apply \cite[Chapter II, Proposition 6.3 and 6.4]{MR1071170} twice to shrink $\Omega_\psi$ to a neighborhood of $(f,f_\gamma)([a,b])$ on which $(\pr_1,\psi)$ is a uniformly bilipschitz continuous map.
  	
  	Since $(f,f_\gamma)([a,b]) \subseteq \Omega$ is compact, there is $R >0$ with 
  	$$ \bigcup_{t \in [a,b]} \{f(t)\} \times B_R (f_\gamma (t)) \subseteq \Omega_\psi$$
  	Using uniform Lipschitz continuity, we can shrink $R>0$ such that 
  	$$(\pr_1,\psi)(\{f(t)\} \times B_R (f_\gamma (t))) \subseteq \{f(t)\} \times B_C (\gamma(t)) \quad \forall t \in [a,b]$$ 
  	Using the norm equivalence of the Riemanian metric with the Banach space norm, we can clearly choose $r>0$ with the desired properties, since $T\kappa^{-1} (\Omega_\psi) \subseteq \Omega$.
  	
  	Observe that once we fix $r>0$, using again the equivalence of norms, there is $L>0$ such that $\{f(t)\} \times B_L (f_\gamma (t)) \subseteq T\kappa(B_r^G (T\kappa^{-1}(f(t), f_\gamma (t))))$ for all $t \in [a,b]$. Using now Lipschitz continuity of $(\pr_1, \psi)^{-1}$, we can choose $s>0$ such that for all $t \in [a,b]$ we have $\{f(t)\} \times B_s (\gamma (t)) \subseteq (\pr_1, \psi) ( \{f(t)\} \times B_L (f_\gamma (t)))$. This $s$ satisfies the assertion of the Lemma. 
  \end{proof}

 \end{appendix}
\addcontentsline{toc}{section}{References}
\bibliographystyle{new}
\bibliography{SRVT_abscont}

\begin{thebibliography}{BBMM14}
\providecommand{\url}[1]{\texttt{#1}}
\providecommand{\urlprefix}{URL }
\expandafter\ifx\csname urlstyle\endcsname\relax
  \providecommand{\doi}[1]{doi:\discretionary{}{}{}#1}\else
  \providecommand{\doi}{doi:\discretionary{}{}{}\begingroup
  \urlstyle{rm}\Url}\fi
\providecommand{\eprint}[2][]{\url{#2}}

\bibitem[AB06]{MR2378491}
Aliprantis, C.~D. and Border, K.~C.
\newblock \emph{Infinite dimensional analysis} (Springer, Berlin, 2006), third
  edn.
\newblock A hitchhiker's guide

\bibitem[Ama90]{MR1071170}
Amann, H.
\newblock \emph{Ordinary differential equations}, \emph{de Gruyter Studies in
  Mathematics}, vol.~13 (Walter de Gruyter \& Co., Berlin, 1990).
\newblock An introduction to nonlinear analysis, Translated from the German by
  Gerhard Metzen

\bibitem[AS15]{MR3342623}
Alzaareer, H. and Schmeding, A.
\newblock \emph{Differentiable mappings on products with different degrees of
  differentiability in the two factors}.
\newblock Expo. Math. \textbf{33} (2015)(2):184--222

\bibitem[BBM14a]{MR3265197}
Bauer, M., Bruveris, M. and Michor, P.~W.
\newblock \emph{Homogeneous {S}obolev metric of order one on diffeomorphism
  groups on real line}.
\newblock J. Nonlinear Sci. \textbf{24} (2014)(5):769--808

\bibitem[BBM14b]{bauer_overview_2014}
Bauer, M., Bruveris, M. and Michor, P.~W.
\newblock \emph{Overview of the {Geometries} of {Shape} {Spaces} and
  {Diffeomorphism} {Groups}}.
\newblock Journal of Mathematical Imaging and Vision  (2014):1--38

\bibitem[BBMM14]{bauer14cri}
Bauer, M., Bruveris, M., Marsland, S. and Michor, P.~W.
\newblock \emph{Constructing reparameterization invariant metrics on spaces of
  plane curves}.
\newblock Differential Geom. Appl. \textbf{34} (2014):139--165

\bibitem[Bru15]{bruveris2015}
Bruveris, M.
\newblock \emph{{Optimal reparameterization in the square root velocity
  framework}} 2015.
\newblock \urlprefix\url{http://arxiv.org/abs/1507.02728v2}.
\newblock {a}rXiv:1507.02728v2, to appear in SIAM J.\ Math.\ Anal.

\bibitem[CES16]{celledoni15sao}
Celledoni, E., Eslitzbichler, M. and Schmeding, A.
\newblock \emph{Shape analysis on {L}ie groups with applications in computer
  animation}.
\newblock J. Geom. Mech. \textbf{8} (2016)(3):273--304

\bibitem[Eng89]{MR1039321}
Engelking, R.
\newblock \emph{General topology}, \emph{Sigma Series in Pure Mathematics},
  vol.~6 (Heldermann Verlag, Berlin, 1989), second edn.
\newblock Translated from the Polish by the author

\bibitem[FK72]{FlaKli72}
{Flaschel}, P. and {Klingenberg}, W.
\newblock \emph{{Riemannsche Hilbertmannigfaltigkeiten. Periodische
  Geod\"atische. Mit einem Anhang von H. Karcher.}}
\newblock {Lecture Notes in Mathematics. 282. Berlin-Heidelberg-New York:
  Springer-Verlag.} 1972

\bibitem[Gl{\"o}02]{MR1934608}
Gl{\"o}ckner, H.
\newblock \emph{Lie group structures on quotient groups and universal
  complexifications for infinite-dimensional {L}ie groups}.
\newblock J. Funct. Anal. \textbf{194} (2002)(2):347--409

\bibitem[Gl{\"o}04]{hg2004c}
Gl{\"o}ckner, H.
\newblock \emph{Lie groups over non-discrete topological fields} 2004.
\newblock \eprint{arXiv:math/0408008v1}

\bibitem[Gl{\"o}15]{glockner15fos}
Gl{\"o}ckner, H.
\newblock \emph{Fundamentals of submersions and immersions between
  infinite-dimensional manifolds}.
\newblock arXiv:1502.05795v3 [math]  (2015).
\newblock \urlprefix\url{http://arxiv.org/abs/1208.0715}

\bibitem[Gl{\"o}16]{hgmeasure16}
Gl{\"o}ckner, H.
\newblock \emph{{Measurable regularity properties of infinite-dimensional Lie
  groups}} 2016.
\newblock \urlprefix\url{http://arxiv.org/abs/1601.02568v1}.
\newblock {a}rXiv:1601.02568v1

\bibitem[JSS14]{su14sao}
Jingyong~Su, Sebastian~Kurtek, E.~K. and Srivastava, A.
\newblock \emph{Statistical analysis of trajectories on {R}iemmannian
  manifolds: bird migration, hurricane tracking and video surveillance}.
\newblock The Annals of Applied Statistics \textbf{8} (2014)(2):530--552

\bibitem[Kli95]{MR1330918}
Klingenberg, W. P.~A.
\newblock \emph{Riemannian geometry}, \emph{de Gruyter Studies in Mathematics},
  vol.~1 (Walter de Gruyter \& Co., Berlin, 1995), second edn.

\bibitem[KM97]{MR1471480}
Kriegl, A. and Michor, P.~W.
\newblock \emph{The convenient setting of global analysis}, \emph{Mathematical
  Surveys and Monographs}, vol.~53 (American Mathematical Society, Providence,
  RI, 1997)

\bibitem[Lan99]{MR1666820}
Lang, S.
\newblock \emph{Fundamentals of differential geometry}, \emph{Graduate Texts in
  Mathematics}, vol. 191 (Springer-Verlag, New York, 1999)

\bibitem[LBAB15]{MR3442194}
Le~Brigant, A., Arnaudon, M. and Barbaresco, F.
\newblock \emph{Reparameterization invariant metric on the space of curves}.
\newblock In \emph{Geometric science of information}, \emph{Lecture Notes in
  Comput. Sci.}, vol. 9389, pp. 140--149 (Springer, Cham, 2015)

\bibitem[Mic80]{MR583436}
Michor, P.~W.
\newblock \emph{Manifolds of differentiable mappings}, \emph{Shiva Mathematics
  Series}, vol.~3 (Shiva Publishing Ltd., Nantwich, 1980)

\bibitem[Mic08]{MR2428390}
Michor, P.~W.
\newblock \emph{Topics in differential geometry}, \emph{Graduate Studies in
  Mathematics}, vol.~93 (American Mathematical Society, Providence, RI, 2008)

\bibitem[Pet06]{MR2243772}
Petersen, P.
\newblock \emph{Riemannian geometry}, \emph{Graduate Texts in Mathematics},
  vol. 171 (Springer, New York, 2006), second edn.

\bibitem[Sch97]{MR1417259}
Schechter, E.
\newblock \emph{Handbook of analysis and its foundations} (Academic Press,
  Inc., San Diego, CA, 1997)

\bibitem[SKJJ11]{srivastava11sao}
Srivastava, A., Klassen, E., Joshi, S. and Jermyn, I.
\newblock \emph{Shape analysis of elastic curves in euclidean spaces}.
\newblock Pattern Analysis and Machine Intelligence, IEEE Transactions on
  \textbf{33} (2011):1415--1428

\bibitem[Sta08]{0809.3104v1}
Stacey, A.
\newblock \emph{{How to Construct a Dirac Operator in Infinite Dimensions}}
  2008.
\newblock \urlprefix\url{http://arxiv.org/abs/0809.3104}.
\newblock \eprint{0809.3104v1}

\bibitem[SW15]{MR3351079}
Schmeding, A. and Wockel, C.
\newblock \emph{The {L}ie group of bisections of a {L}ie groupoid}.
\newblock Ann. Global Anal. Geom. \textbf{48} (2015)(1):87--123

\end{thebibliography}

\end{document}